\documentclass[11pt,epsfig,amsfonts]{amsart}

\usepackage{mathrsfs} 

\usepackage{epsfig}
\usepackage{amsmath}
\usepackage{amssymb}
\usepackage{amscd}
\usepackage{graphicx}
\usepackage{color}
\usepackage[english]{babel}

\usepackage{cite}
\usepackage{mathrsfs}
\usepackage{amssymb}
\usepackage{graphicx,latexsym,amsmath,amssymb,amsthm,enumerate}

\newtheorem{theorem}{Theorem}[section]
\newtheorem{lemma}{Lemma}[section]
\newtheorem{corollary}{Corollary}[section]

\newtheorem{definition}{Definition}[section]
\newtheorem{remark}{Remark}[section]

\newcommand{\be}{\begin{equation}}
\newcommand{\ee}{\end{equation}}
\newcommand{\ba}{\begin{align}}
\newcommand{\ea}{\end{align}}

\numberwithin{equation}{section}

\begin{document}

\baselineskip =1.4\baselineskip
\thanks{*Corresponding author}
\thanks{The project is supported by the National Natural Science Foundation of China
(Grant no.12171343 and no.12571183) and Sichuan Science and Technology Program(no.2022ZYD0009 and no.2022JDTD0019).}

\author{Sheng Wang} \address[Sheng Wang]
{School of Mathematical Science \\  Sichuan Normal University   \\
 Chengdu, Sichuan 610066, P.R. China}
   \email[]{wangsmath@163.com}

\author{Juan Huang$^{\ast}$} \address[Juan Huang]
{School of Mathematical Science \\  Sichuan Normal University   \\
 Chengdu, Sichuan 610066, P.R. China}
   \email[]{hjmath@163.com}

\title[Existence and qualitative properties of ground state solutions for the Schr\"{o}dinger-Bopp-Podolsky system] {Existence and qualitative properties of ground state solutions for the Schr\"{o}dinger-Bopp-Podolsky system}

\pagestyle{plain}

\begin{abstract}
{This paper concerns the existence and related properties of solutions to the Schr\"{o}dinger-Bopp-Podolsky system, which reduces to a nonlinear and nonlocal partial differential equation describing a Schr\"{o}dinger field coupled with its electromagnetic field in Bopp-Podolsky theory under purely electrostatic conditions. Firstly, by applying the mountain-pass lemma, we obtain the existence of nontrivial solutions. Then, through some estimates of the ground state energy, we prove the existence of ground state solutions. By exploring the relationship between solutions and paths associated with critical points, we further demonstrate that the obtained solutions are ground states of mountain-pass type. Additionally, the positivity, radial symmetry, rotational invariance, and exponential decay of the ground state solutions are considered. Finally, in the radial case, we explore the asymptotic behavior of the obtained solutions with respect to $a$.}
\end{abstract}

\maketitle
\tableofcontents

\renewcommand{\theequation}
{\thesection.\arabic{equation}}
\setcounter{equation}{0}
\section{Introduction} \noindent

In this paper, we are interested in the following Schr\"{o}dinger-Bopp-Podolsky system:
\begin{equation}\label{main}
\left\{
\begin{aligned}
&-\Delta u+\omega u+\lambda q^2\phi u=|u|^{p-2}u \quad \hbox{in}~~\mathbb{R}^3,\\
&-\Delta\phi+a^2\Delta^2\phi=4\pi |u|^2 \quad \hbox{in}~~\mathbb{R}^3,
\end{aligned}
\right.
\end{equation}
where $u: \mathbb{R}^3\rightarrow\mathbb{C}$, $\phi: \mathbb{R}^3\rightarrow\mathbb{R}$, $\omega\in\mathbb{R}$, $\lambda\in\mathbb{R}$, $q\neq 0$, $a>0$ and $p\in(2, 6]$. In system (\ref{main}), from the perspective of physics, $|u|$ represents the modulus of the wave function, and $\phi$ denotes the electrostatic potential. As we all know, the Bopp-Podolsky theory, a second-order gauge theory of the electromagnetic field, was developed independently by Bopp \cite{Bopp1940} and then by Podolsky \cite{Podolsky1942}. According to the Mie theory \cite{Mie1913} and its generalizations proposed by Born and Infeld \cite{Born1933, Born1934, BornI1933, BornI1934}, the Bopp-Podolsky theory was introduced to solve the so-called infinity problem in classical Maxwell theory. In addition, the Bopp-Podolsky theory can be regarded as an effective theory for short distances, while at large distances it is experimentally indistinguishable from the Maxwell theory, see \cite{Aberqi2022, Farid2023, Frenkel1996}.

In recent years, variational methods have become a pivotal tool for probing the existence of nontrivial solutions to system (\ref{main}) under the condition $\lambda>0$. When examining $\omega\in\mathbb{R}$ in system (\ref{main}), two distinct stances emerge: $\omega$ can either be regarded as a pre-defined constant or remain an unknown quantity.
For the first case, d'Avenia and Siciliano \cite{dAvenia2019} showed that system (\ref{main}) possesses nontrivial positive solutions by using the splitting lemma and the monotonicity trick, where $p\in(2, 6)$ with $|q|$ sufficiently small, or $p\in(3, 6)$ and $q\neq0$. Alongside this, they investigated the asymptotic behavior of solutions as $a\rightarrow0$. Moreover, by applying the Pohozaev identity, they were able to conclude that the system does not admit any nontrivial solution for $p\geq6$. Subsequently, Silva and Siciliano \cite{Silva2020} utilized the so-called ``fibering approach" and found that system (\ref{main}) admits no solution for sufficiently large $q$ yet possesses two radial solutions for sufficiently small $q$. Furthermore, they gave some qualitative properties of the energy level of the solutions and a variational characterization of the extremal values of $q$. In the latter case, it is natural to prescribe the value of the mass so that $\omega$ can be interpreted as a Lagrange multiplier, and the solutions to system (\ref{main}) are then called normalized solutions. To the best of our knowledge, the existence and multiplicity of normalized solutions for the Schr\"{o}dinger-Bopp-Podolsky system were first studied by Afonso and Siciliano \cite{Afonso2023} in a bounded smooth domain. For the whole space $\mathbb{R}^3$, Liu \cite{LiuC2022} proved the existence of normalized solutions by the minimization method for system (\ref{main}) with $p\in(2, \frac83]$. And then, Ramos and Siciliano \cite{Ramos2023} applied standard scaling arguments and some basic properties of the energy functional to extend the range of $p$ to $(2, 3)\cup(3, \frac{10}3)$. For the $L^2$-supercritical case, i.e. $p\in(\frac{10}3, 6)$, the present authors \cite{Huang2025} proved the existence of normalized solutions to system (\ref{main}) via the mountain-pass lemma. In addition, recently, some papers have considered the existence of nontrivial solutions or ground state solutions for system (\ref{main}) when the potential $V(x)$ is allowed to be not constant or the nonlinearity $|u|^{p-2}u$ is of more general form, namely,
\begin{equation}\nonumber
\left\{
\begin{aligned}
&-\Delta u+V(x)u+\lambda\phi u=f(u) \quad \hbox{in}~~\mathbb{R}^3,\\
&-\Delta\phi+a^2\Delta^2\phi=4\pi |u|^2 \quad \hbox{in}~~\mathbb{R}^3,
\end{aligned}
\right.
\end{equation}
for $\lambda>0$. For more details, we refer the reader to \cite{Silva2020JDE, Teng2021, Yang2020} for the subcritical case and \cite{Chen2020, Damian2024, Li2020, Liu2022, Zhu2021} for the critical nonlinearity.

It should be noted that all the above-mentioned results were obtained under the condition of $\lambda>0$. For system (\ref{main}) with $\lambda<0$, by applying the truncation argument \cite{Azorero1991} and the abstract minimax theorem \cite{Jeanjean2019}, Peng \cite{Peng2024} proved that when $\lambda<0$ and $p=6$, system (\ref{main}) admits at least $n$ pairs of radial normalized solutions with negative energy. However, when $\omega$ is treated as a known quantity and $\lambda<0$, the existence of solutions to system (\ref{main}) remains unknown. In view of this, the present paper aims to address this problem. Specifically speaking, we study the existence of nontrivial solutions and ground state solutions for the following Schr\"{o}dinger-Bopp-Podolsky system:
\begin{equation}\label{main6}
\left\{
\begin{aligned}
&-\Delta u+\omega u-\mu\phi u=|u|^{p-2}u \quad \hbox{in}~~\mathbb{R}^3,\\
&-\Delta\phi+a^2\Delta^2\phi=4\pi |u|^2 \quad \hbox{in}~~\mathbb{R}^3,
\end{aligned}
\right.
\end{equation}
where $\omega>0$, $\mu:=-\lambda q^2>0$, $a>0$ and $p\in(2, 6]$. Furthermore, we are concerned with some qualitative properties of the ground state solutions, such as positivity, radial symmetry, rotational invariance, and exponential decay. Finally, we also investigate the asymptotic behavior of the solutions.

Before proceeding with our research, a few preliminaries are in order. We introduce here the space $\mathcal{D}$ as the completion of $C^{\infty}_c(\mathbb{R}^3)$ with respect to the norm ${(\|\nabla\phi\|^2_{L^2(\mathbb{R}^3)}+a^2\|\Delta\phi\|^2_{L^2(\mathbb{R}^3)})}^{1\over2}$; for more details on this space, see Section 2. For fixed $a>0$, we say that a pair $(u, \phi)\in H^1(\mathbb{R}^3)\times\mathcal{D}$ is a solution of system (\ref{main6}) if
\begin{align*}
\int_{\mathbb{R}^3}\nabla u\cdot\nabla vdx+\omega\int_{\mathbb{R}^3}uvdx-\mu\int_{\mathbb{R}^3}\phi uvdx=\int_{\mathbb{R}^3}|u|^{p-2}uvdx, ~~~\forall~ v\in H^1(\mathbb{R}^3),
\end{align*}
\begin{align*}
\int_{\mathbb{R}^3}\nabla\phi\cdot\nabla\xi dx+a^2\int_{\mathbb{R}^3}\Delta\phi\Delta\xi dx=4\pi\int_{\mathbb{R}^3}\xi|u|^2dx, ~~\forall~ \xi\in\mathcal{D}.
\end{align*}
Furthermore, we call that a solution $(u, \phi)$ is nontrivial if $u\not\equiv 0$. As described in \cite{dAvenia2019}, by the Riesz Theorem, for every fixed $u\in H^1(\mathbb{R}^3)$, there exists a unique solution $\phi_u\in\mathcal{D}$ of the second equation in system (\ref{main6}), which is given by
\begin{align}\label{main-phi}
\phi_u:=\mathcal{K}\ast |u|^2,
\end{align}
where $\ast$ represents the convolution in $\mathbb{R}^3$ and $\mathcal{K}(x)=\frac{1-e^{-\frac{|x|}{a}}}{|x|}$. Inserting (\ref{main-phi}) into the first equation of system (\ref{main6}), we get
\begin{equation}\label{main2}
-\Delta u+\omega u-\mu\phi_u u=|u|^{p-2}u \quad\hbox{in}~ \mathbb{R}^3.
\end{equation}
Then, the following statements are equivalent:
\begin{itemize}
 \item[(i)] the pair $(u, \phi_u)\in H^1(\mathbb{R}^3)\times\mathcal{D}$ is a solution of system (\ref{main6});
 \item[(ii)] $u\in H^1(\mathbb{R}^3)$ is a solution of Eq.(\ref{main2}).
\end{itemize}
Clearly, the solution $u$ corresponds to the critical point of the following energy functional:
\begin{align}\label{energy}
J(u):=\frac12\|\nabla u\|^2_{L^2(\mathbb{R}^3)}+\frac{\omega}2\|u\|^2_{L^2(\mathbb{R}^3)}-\frac{\mu}4\int_{\mathbb{R}^3}\phi_u|u|^2dx-\frac1{p}\|u\|^p_{L^p(\mathbb{R}^3)}.
\end{align}
Moreover, in view of \cite{dAvenia2019}, the functional $J$ is $C^1(H^1(\mathbb{R}^3), \mathbb{R})$ and, for all $u, v\in H^1(\mathbb{R}^3)$
\begin{align}\label{Nehari-functional}
\langle J'(u), v\rangle=\int_{\mathbb{R}^3}\nabla u\cdot\nabla vdx+\omega\int_{\mathbb{R}^3}uvdx-\mu\int_{\mathbb{R}^3}\phi_u uvdx-\int_{\mathbb{R}^3}|u|^{p-2}uvdx.
\end{align}

To find a nontrivial solution of Eq.(\ref{main2}), we first note that owing to the different spatial scaling properties of each term in Eq.(\ref{main2}), the usual constrained minimization and Lagrange multiplier methods are not applicable here. Following the strategy in \cite{Jeanjean1997, Li2011}, we aim to construct a solution by virtue of the mountain-pass lemma. For this purpose, we define the mountain-pass level $c$ as
\begin{align}\label{mountain-pass}
c:=\inf\limits_{\gamma\in\Gamma}\max\limits_{t\in[0, 1]}J(\gamma(t)),
\end{align}
where
\begin{align}\label{mountain-pass1}
\Gamma:=\{\gamma\in C([0, 1], H^1(\mathbb{R}^3)):~ \gamma(0)=0, J(\gamma(1))<0\}.
\end{align}
In view of the interaction between local and nonlocal terms, we construct a Cerami sequence $\{u_n\}^{\infty}_{n=1}$ at the level $c$ and then show that, up to translations, this sequence converges to a nontrivial solution $u$ of Eq.(\ref{main2}). By analyzing the properties of the Cerami sequence constructed above, we can establish the existence of nontrivial solutions for Eq.(\ref{main2}) as follows.
\begin{theorem}\label{theorem1}
Eq.(\ref{main2}) admits a nontrivial solution $u\in H^1(\mathbb{R}^3)$ if one of the following conditions is satisfied:
	\begin{itemize}
		\item [(1)] $p\in(2, 6)$ and $\mu>0$;
		\item [(2)] $p=6$ and $\mu>0$ large enough.
	\end{itemize}
In addition, $u\in W^{2, s}(\mathbb{R}^3)$ for every $s>1$.
\end{theorem}

\begin{remark}\label{remark1}
Compared with the related results in \cite{dAvenia2019}, our proofs exhibit several differences: (a) In order to obtain the existence of nontrivial solutions, the authors in \cite{dAvenia2019} constructed a Pohozaev-Palais-Smale sequence. However, due to the fact that Eq.(\ref{main2}) has a negative nonlocal term, the general Palais-Smale sequence is not suitable since verifying boundedness and compactness is challenging. For this reason, we adopted the strategy in \cite{Jeanjean1997, Li2011} to construct a Cerami sequence at the mountain-pass level $c$. (b) The authors of \cite{dAvenia2019} further employed the Pohozaev identity to prove the nonexistence of nontrivial solutions when $p\geq 6$, which contrasts sharply with our conclusion in Theorem \ref{theorem1}. Conversely, in our setting, by demonstrating that the mountain-pass level $c$ satisfies $0<c<\frac13 K^{\frac32}$ (see Lemma \ref{proof-theorem1-lemma1}), and combining the splitting lemma with some new analytical techniques, we established the existence of nontrivial solutions to Eq.(\ref{main2}) at the Sobolev critical exponent (i.e., $p=6$).
\end{remark}

In general, we cannot determine whether the nontrivial solution $u$ in Theorem \ref{theorem1} is a ground state. Here, a ground state refers to a solution $u\in H^1(\mathbb{R}^3)$ of Eq.(\ref{main2}) satisfying
\begin{align}\label{ground-state}
c_g:=J(u)=\inf\limits_{v\in\mathcal{N}}J(v).
\end{align}
Here,
\begin{align}\label{Nehari-manifold}
\mathcal{N}:=\{u\in H^1(\mathbb{R}^3)\setminus\{0\}:~ \langle J'(u), u\rangle=0\}.
\end{align}
Inspired by Jeanjean and Tanaka \cite{Jeanjean2003}, any solution of Eq.(\ref{main2}) can be lifted to a path associated with critical points. This approach allows us to estimate the critical level $J(u)$, thereby proving that $u$ is a ground state of mountain-pass type. Notably, in proving the existence of ground state solutions, we establish a key inequality involving $J(u)$, $J(\tau u)$ and $\langle J'(u), u\rangle$ (see Lemma \ref{preliminary-lemma-1}) to address the lack of compactness caused by the unbounded domain, this inequality is crucial to our arguments. With it, we can then recover the compactness of the minimizing sequence $\{u_n\}^{\infty}_{n=1}$ and demonstrate its weak convergence to a nontrivial limit point $u\in H^1(\mathbb{R}^3)$. Finally, by applying the deformation lemma and intermediate value theorem, we prove that $u$ is a critical point of $J$, and hence is a ground state solution of Eq.(\ref{main2}).

\begin{theorem}\label{theorem2}
Let $p\in[4, 6]$ and satisfy the conditions in Theorem \ref{theorem1}, then Eq.(\ref{main2}) has a ground state solution $u\in H^1(\mathbb{R}^3)$ such that
$$c_g=J(u)=\inf\limits_{v\in\mathcal{N}}J(v)=\inf\limits_{v\in H^1(\mathbb{R}^3)\setminus\{0\}}\max\limits_{\tau>0}J(\tau v)>0,$$
and $c_g=c$.
\end{theorem}

Building on Theorem \ref{theorem2}, we further investigate several qualitative properties of these solutions, including positivity, radial symmetry, rotational invariance, and exponential decay. 

\begin{theorem}\label{theorem4}
	Assume $p\in[4, 6]$ and satisfy the conditions in Theorem \ref{theorem1}. Let $u\in H^1(\mathbb{R}^3)$ be a ground state solution to Eq.(\ref{main2}), then
	\begin{itemize}
		\item [(1)] $|u|>0$ is a ground state to Eq.(\ref{main2}) for all $x\in\mathbb{R}^3$.
		\item [(2)] There exist $x_0\in\mathbb{R}^3$ and a nonincreasing positive function $v: (0, +\infty)\rightarrow\mathbb{R}$ such that $|u(x)|=v(|x-x_0|)$ for almost every $x\in\mathbb{R}^3$.
		\item [(3)] $u=e^{i\theta}|u|$ for some $\theta\in\mathbb{R}$.
		\item [(4)] There exists a constant $M>0$ such that
		$$|u(x)|\leq Me^{-\frac{\sqrt{\omega}}2|x|}, \quad\hbox{for all}~ x\in\mathbb{R}^3.$$
	\end{itemize}
\end{theorem}

It is particularly emphasized that when proving the exponential decay of the ground state solutions, the radial lemma \cite{Berestycki1983} cannot be directly applied here due to the influence of the nonlocal term. Thus, the main challenge we face is to verify that $|u(x)|\rightarrow0$ as $|x|\rightarrow\infty$. To address this difficulty, we adopt a new technique in this work, namely the Moser iteration technique, which is mainly inspired by \cite{Gilbarg1983, Kang2024}.

In the final part, we study the asymptotic behavior of solutions with respect to $a$. From the point of view of physics, the Bopp-Podolsky parameter $a>0$, which has inverse mass dimension, can be interpreted as a cut-off distance or linked to an effective electron radius. For more physical details, refer to \cite{Bertin2017, Farid2023} and the references therein. To present the theorem clearly, we denote $(u^a, \phi^a)\in H^1(\mathbb{R}^3)\times\mathcal{D}$ as solutions of Eq.(\ref{main2}). Let us observe that, due to the invariance of $J$ under the group induced by the action of rotations on $H^1(\mathbb{R}^3)$, we can restrict our study to $H^1_r(\mathbb{R}^3)$, the subspace of radial functions, which is a natural constraint: if $u\in H^1_r(\mathbb{R}^3)$ is a critical point of $J|_{H^1_r(\mathbb{R}^3)}$, then it is a critical point for the functional on the entire $H^1(\mathbb{R}^3)$. Then the same results as in Theorems \ref{theorem1} and \ref{theorem2} hold in the radial setting.

\begin{theorem}\label{theorem3}
	For each $a>0$, let $(u^a, \phi^a)\in H^1_r(\mathbb{R}^3)\times\mathcal{D}_r$ be solution of Eq.(\ref{main2}). Then as $a\rightarrow0$, up to a subsequence, there exist $(u^0, \phi^0)\in H^1_r(\mathbb{R}^3)\times D^{1, 2}_r(\mathbb{R}^3)$ such that
	$$(u^a, \phi^a)\rightarrow(u^0, \phi^0) \quad\hbox{in}~ H^1_r(\mathbb{R}^3)\times D^{1, 2}_r(\mathbb{R}^3),$$
	where $(u^0, \phi^0)\in H^1_r(\mathbb{R}^3)\times D^{1, 2}_r(\mathbb{R}^3)$ is a solution of the following Schr\"{o}dinger-Poisson-Slater equation
	\begin{equation}\label{main3}
		-\Delta u+\omega u-\mu\left(\frac1{|x|}\ast|u|^2\right)=|u|^{p-2}u \quad \hbox{in}~~\mathbb{R}^3,
	\end{equation}
	and $\phi^a:=\mathcal{K}\ast|u^a|^2$, $\phi^0:=\frac1{|x|}\ast|u^0|^2$.
\end{theorem}

This paper is organized as follows. In Section 2, we present general preliminaries related to our problem. In Section 3, we establish a Cerami sequence $\{u_n\}^{\infty}_{n=1}$ for the functional $J$, and thereby prove the existence of a nontrivial solution to Eq.(\ref{main2}). In Section 4, we give the minimax property of ground state energy for $J$ and proceed with the analysis of the qualitative properties of the ground state solutions to Eq.(\ref{main2}). In Section 5, we study the behavior of the radial solutions with respect to $a$.

Throughout the paper we make use of the following notations:
\begin{itemize}
 \item $C^{\infty}_c$ denotes the space of the functions infinitely dofferentiable with compact support in $\mathbb{R}^3$;
 \item $\|\cdot\|_{L^r}$ denotes the usual norm of $L^r(\mathbb{R}^3)$ for $r\in [1, \infty)$, $\|\cdot\|_{L^\infty}$ denotes the norm of $L^{\infty}(\mathbb{R}^3)$, and $\|\cdot\|_{H^1}$ denotes the usual norm of $H^1(\mathbb{R}^3)$;
 \item we use the symbol $o_n(1)$ for a vanishing sequence in the specified space;
 \item if not specified, the domain of the integrals is $\mathbb{R}^3$;
 \item we use $C, C_1, C_2, \cdots$ to denote suitable positive constants whose value may also change
     from line to line;
 \item $B_R(x_0)$ denotes the ball centered at $x_0$ with radius $R>0$.
\end{itemize}

\renewcommand{\theequation}
{\thesection.\arabic{equation}}
\setcounter{equation}{0}
\section{Preliminaries}

In this section, we recall some preliminary results that will be used later. First, let us introduce several inequalities.

\begin{lemma}[Hardy-Littlewood-Sobolev inequality,  \cite{Lieb2001}]\label{HLS}
Let $s, r>1$ and $0<\alpha<3$ with $\frac1s+\frac1r+\frac{\alpha}3=2$. For $u\in L^s(\mathbb{R}^3)$, $v\in L^r(\mathbb{R}^3)$, there exists a constant $C(s, r)$ such that
$$\iint\frac{|u(x)v(y)|}{|x-y|^{\alpha}}dxdy\leq C(s, r)\|u\|_{L^s}\|v\|_{L^r}.$$
\end{lemma}

Particularly, if $s=r=\frac65$, then the Hardy-Littlewood-Sobolev inequality implies that
$$\int(\mathcal{K}\ast|u|^2)|u|^2dx\leq\int\left(\frac1{|x|}\ast|u|^2\right)|u|^2dx\leq C\|u\|^4_{L^{\frac{12}5}}.$$

\begin{lemma}[\cite{Berestycki1983}]\label{radial-inequality}
If $u\in L^s(\mathbb{R}^3)$, $1\leq s<+\infty$, is a radial nonincreasing function (i.e. $0\leq u(x)\leq u(y)$ if $|x|\geq|y|$), then one has
$$|u(x)|\leq|x|^{-\frac3{s}}\left(\frac3{|S^2|}\right)^{\frac1{s}}\|u\|_{L^s}, \quad x\neq0,$$
where $|S^2|$ is the area of the unit sphere in $\mathbb{R}^3$.
\end{lemma}

The following lemma can be found in Br\'{e}zis and Kato \cite{Brezis-Kato1979}, which is crucial for the $L^{\infty}$-estimate for the solution of Eq.(\ref{main2}).
\begin{lemma}\label{L-infty}
Let $k\in L^{\frac32}(\mathbb{R}^3)$ be a nonnegative function. Then for every $\varepsilon>0$, there exists a constant $C(\varepsilon, k)>0$ such that
\begin{align*}
\int k(x)|u|^2dx\leq\varepsilon\int|\nabla u|^2dx+C(\varepsilon, k)\int|u|^2dx, ~~ \hbox{for all}~ u\in H^1(\mathbb{R}^3).
\end{align*}
\end{lemma}

Denote by $K$ the best Sobolev constant as follows:
\begin{align}\label{best-constant}
K=\inf\limits_{u\in D^{1, 2}(\mathbb{R}^3)\setminus\{0\}}\frac{\int|\nabla u|^2dx}{(\int|u|^6dx)^{\frac13}},
\end{align}
where $D^{1, 2}(\mathbb{R}^3):=\{u: u\in L^6(\mathbb{R}^3), \nabla u\in L^2(\mathbb{R}^3)\}$ is a Banach space endowed with the norm $\|u\|_{D^{1, 2}(\mathbb{R}^3)}=\left(\int|\nabla u|^2dx\right)^{\frac12}$. It is well-known that the embedding
$D^{1, 2}(\mathbb{R}^3)\hookrightarrow L^6(\mathbb{R}^3)$ is continuous. Let $\mathcal{D}$ be the completion of $C^{\infty}_c(\mathbb{R}^3)$ with respect to the norm $\|\cdot\|_{\mathcal{D}}$ induced by the scalar product
\begin{align*}
\langle\varphi, \psi\rangle_{\mathcal{D}}:=\int\nabla\varphi\cdot\nabla\psi dx+a^2\int\Delta\varphi\Delta\psi dx.
\end{align*}
Then $\mathcal{D}$ is a Hilbert space continuously embedded into $D^{1, 2}(\mathbb{R}^3)$ and consequently in $L^6(\mathbb{R}^3)$. The characterization of the space $\mathcal{D}$ can be stated as follows.

\begin{lemma}[\cite{dAvenia2019}]\label{continuously-embedded}
The Hilbert space $\mathcal{D}$ is continuously embedded in $L^{\infty}(\mathbb{R}^3)$.
\end{lemma}

\begin{lemma}[\cite{dAvenia2019}]\label{dense}
The space $C^{\infty}_c(\mathbb{R}^3)$ is dense in $\mathcal{A}$, where
$$\mathcal{A}:=\{\phi\in D^{1, 2}(\mathbb{R}^3):~ \Delta\phi\in L^2(\mathbb{R}^3)\}$$
normed by $\sqrt{\langle\phi, \phi\rangle_{\mathcal{D}}}$ and, consequently $\mathcal{D}=\mathcal{A}$.
\end{lemma}

As \cite{dAvenia2019} indicates, the next fundamental properties hold.

\begin{lemma}[\cite{dAvenia2019}]\label{property1}
For all $y\in\mathbb{R}^3$, $\mathcal{K}(\cdot-y)$ solves in the sense of distributions
$$-\Delta\phi+a^2\Delta^2\phi=4\pi\delta_y.$$
Moreover,
\begin{itemize}
 \item[(1)] if $f\in L^1_{loc}(\mathbb{R}^3)$ and, for a.e. $x\in\mathbb{R}^3$, the map $y\in\mathbb{R}^3\mapsto\frac{f(y)}{|x-y|}$ is summable, then $\mathcal{K}\ast f\in L^1_{loc}(\mathbb{R}^3)$;
 \item[(2)] if $f\in L^s(\mathbb{R}^3)$ with $1\leq s<\frac32$, then $\mathcal{K}\ast f\in L^q(\mathbb{R}^3)$ for $q\in(\frac{3s}{3-2s}, +\infty]$.
\end{itemize}
In both cases $\mathcal{K}\ast f$ solves
$$-\Delta\phi+a^2\Delta^2\phi=4\pi f$$
in the sense of distributions, and we have the following distributional derivatives
$$\nabla(\mathcal{K}\ast f)=(\nabla\mathcal{K})\ast f \quad\hbox{and}\quad \Delta(\mathcal{K}\ast f)=(\Delta\mathcal{K})\ast f \quad\hbox{a.e. in}~~ \mathbb{R}^3.$$
\end{lemma}

Denote
$$\phi_u:=\mathcal{K}\ast|u|^2=\int\frac{1-e^{-\frac{|x-y|}a}}{|x-y|}|u(y)|^2dy,$$
then we have the following useful properties.

\begin{lemma}[\cite{dAvenia2019}]\label{property2}
For every $u\in H^1({\mathbb{R}^3})$ we have:
\begin{itemize}
 \item[(1)] for every $y\in\mathbb{R}^3$, $\phi_{u(\cdot+y)}=\phi_u(\cdot+y)$;
 \item[(2)] $\phi_u\geq 0$;
 \item[(3)] for every $s\in(3, +\infty]$, $\phi_u\in L^s(\mathbb{R}^3)\cap C_0(\mathbb{R}^3)$;
 \item[(4)] for every $s\in(\frac32, +\infty]$, $\nabla\phi_u=\nabla\mathcal{K}\ast |u|^2\in L^s(\mathbb{R}^3)\cap C_0(\mathbb{R}^3)$;
 \item[(5)] $\phi_u\in\mathcal{D}$;
 \item[(6)] $\|\phi_u\|_{L^6}\leq C\|u\|^2_{H^1}$;
 \item[(7)] $\phi_u$ is the unique minimizer of the functional
  $$F(\phi)=\frac12\|\nabla\phi\|^2_{L^2}+\frac{a^2}2\|\Delta\phi\|^2_{L^2}-\int\phi |u|^2dx, \quad \phi\in\mathcal{D}.$$
\end{itemize}
Moreover, if $v_n\rightharpoonup v$ in $H^1(\mathbb{R}^3)$, then $\phi_{v_n}\rightharpoonup\phi_v$ in $\mathcal{D}$.
\end{lemma}

\begin{lemma}[\cite{dAvenia2019}]\label{BL}
Let $p\in(2, 6)$ and $\{u_n\}^{\infty}_{n=1}$ be a sequence satisfying that $u_n\rightharpoonup u$ weakly in $H^1(\mathbb{R}^3)$. Then, for any $\varphi\in H^1(\mathbb{R}^3)$,
$$\int\phi_{u_n}u_n\varphi dx\rightarrow\int\phi_u u\varphi dx$$
and
$$\int|u_n|^{p-2}u_n\varphi dx\rightarrow\int|u|^{p-2}u\varphi dx,$$
as $n\rightarrow\infty$.
\end{lemma}

The next lemma is regarding the asymptotic of the nonlocal term, which is helpful to prove the asymptotic behavior of solutions.
\begin{lemma}[\cite{dAvenia2019}]\label{asymptotic-lemma}
Assume $f^0\in L^{\frac65}(\mathbb{R}^3)$, $\{f^a\}_{a\in(0, 1)}\subset L^{\frac65}(\mathbb{R}^3)$. Let
$$\phi^0\in D^{1, 2}(\mathbb{R}^3)~ \hbox{be the unique solution of}~ -\Delta\phi=f^0 ~\hbox{in}~ \mathbb{R}^3$$
and
$$\phi^a\in\mathcal{D} ~\hbox{be the unique solution of}~ -\Delta\phi+a^2\Delta^2\phi=f^a ~\hbox{in}~ \mathbb{R}^3.$$
As $a\rightarrow0$, we have
\begin{itemize}
 \item[(1)] if $f^a\rightharpoonup f^0$ in $L^{\frac65}(\mathbb{R}^3)$, then $\phi^a\rightharpoonup\phi^0$ in $D^{1, 2}(\mathbb{R}^3)$;
 \item[(2)] if $f^a\rightarrow f^0$ in $L^{\frac65}(\mathbb{R}^3)$, then $\phi^a\rightarrow\phi^0$ in $D^{1, 2}(\mathbb{R}^3)$ and $a\Delta\phi^a\rightarrow0$ in $L^2(\mathbb{R}^3)$.
\end{itemize}
\end{lemma}

We end this section by introducing the following lemma.

\begin{lemma}[\cite{Lions1984-2}]\label{vanishing}
Assume that $\{u_n\}_{n=1}^{\infty}$ is bounded sequence in $H^1(\mathbb{R}^3)$ such that
$$\limsup_{n\rightarrow\infty}\int_{B_R(y)}|u_n(x)|^2dx=0$$
for some $R>0$. Then $u_n\rightarrow0$ in $L^r(\mathbb{R}^3)$ for every $r$, with $2<r<6$.
\end{lemma}

\renewcommand{\theequation}
{\thesection.\arabic{equation}}
\setcounter{equation}{0}
\section{Nontrivial solutions} \noindent

In this section, we prove the existence of nontrivial solutions to Eq.(\ref{main2}). Prior to this, we note that, in view of \cite[Appendix A.3]{dAvenia2019}, any critical point $u$ of $J$ satisfies the following Pohozaev equality:
\begin{align}\nonumber
P(u)=&\frac12\|\nabla u\|^2_{L^2}+\frac{3\omega}2\|u\|^2_{L^2}-\frac{5\mu}4\int\phi_u|u|^2dx-\frac3{p}\|u\|^p_{L^p}\\ \label{Pohozaev-equality}
&-\frac{\mu}{4a}\iint e^{-\frac{|x-y|}a}|u(x)|^2|u(y)|^2dxdy=0.
\end{align}
Let
\begin{align}\nonumber
I(u):=&\langle J'(u), u\rangle-P(u)\\ \nonumber
=&\frac12\|\nabla u\|^2_{L^2}-\frac{\omega}2\|u\|^2_{L^2}+\frac{\mu}4\int\phi_u|u|^2dx-\frac{p-3}{p}\|u\|^p_{L^p}\\ \label{Pohozaev-equality1}
&+\frac{\mu}{4a}\iint e^{-\frac{|x-y|}a}|u(x)|^2|u(y)|^2dxdy,
\end{align}
then $I(u)=0$, for all $u\in\mathcal{N}$, where $\mathcal{N}$ is defined in (\ref{Nehari-manifold}).

Under our conditions, since $\phi_u$ depends on the distribution of $u$ over the entire space, for the sequence $\{u_n\}^{\infty}_{n=1}$, the nonlocal term will lead to a correlation between the decay of $\|J'(u_n)\|_{H^{-1}}$ and the growth of the norm of $u_n$. A simple condition that $\|J'(u_n)\|_{H^{-1}}\rightarrow0$ cannot guarantee the boundedness and compactness of the sequence $\{u_n\}^{\infty}_{n=1}$. Therefore, the general Palais-Smale sequences are not applicable, and here, we shall apply a weaker minimax theorem involving Cerami sequences. Now, we first present the definition of a Cerami sequence.

\begin{definition}\label{Cerami-sequence}
Let $(X, \|\cdot\|_X)$ be a real Banach space with its dual space $(X', \|\cdot\|_{X'})$ and $\varphi\in C^1(X, \mathbb{R})$. For $c_0\in\mathbb{R}$, we say that $\{u_n\}^{\infty}_{n=1}\subset X$ is a Cerami sequence, if
$$\varphi(u_n)\rightarrow c_0 ~~\hbox{and}~~ (1+\|u_n\|_{X})\|\varphi'(u_n)\|_{X'}\rightarrow0$$
as $n\rightarrow\infty$.
\end{definition}

\begin{lemma}[Minimax principle, \cite{Li2011}]\label{minimax-principle}
Let $X$ be a Banach space and $M$ a metric space. Let $M_0$ be a closed subspace of $M$ and $\Gamma_0\subset C(M_0, X)$. Define
$$\Gamma_1:=\left\{\gamma_1\in C(M, X):~ \gamma_1\big|_{M_0}\in\Gamma_0\right\}.$$
If $\varphi\in C^1(X, \mathbb{R})$ satisfies
$$\infty>c_0:=\inf\limits_{\gamma_1\in\Gamma_1}\sup\limits_{u\in M}\varphi(\gamma_1(u))>b:=\sup\limits_{\gamma_0\in\Gamma_0}\sup\limits_{u\in M_0}\varphi(\gamma_0(u)),$$
then, for every $\varepsilon\in(0, \frac{c_0-b}2)$, $\delta>0$ and $\gamma_1\in\Gamma_1$ such that
$$\sup\limits_{M}\varphi\circ\gamma_1\leq c_0+\varepsilon,$$
there exists $u\in X$ such that
\begin{itemize}
 \item[(i)] $c_0-2\varepsilon\leq\varphi(u)\leq c_0+2\varepsilon$;
 \item[(ii)] $\hbox{dist}(u, \gamma_1(M))\leq 2\delta$;
 \item[(iii)] $(1+\|u\|_{X})\|\varphi'(u)\|_{X'}\leq\frac{8\varepsilon}{\delta}$.
\end{itemize}
\end{lemma}

Next, we will apply Lemma \ref{minimax-principle} to obtain a Cerami sequence $\{u_n\}^{\infty}_{n=1}$ of $J$ with $I(u_n)\rightarrow0$, as $n\rightarrow\infty$. This idea goes back to \cite{Jeanjean1997}.

\begin{lemma}\label{Cerami-sequence1}
Let $p\in(2, 6]$. Then there exists a sequence $\{u_n\}^{\infty}_{n=1}\subset H^1(\mathbb{R}^3)$ satisfying
\begin{align}\label{Cerami1}
J(u_n)\rightarrow c>0, ~~ (1+\|u_n\|_{H^1})\|J'(u_n)\|_{H^{-1}}\rightarrow0 ~\hbox{and}~ I(u_n)\rightarrow0, ~\hbox{as}~ n\rightarrow\infty,
\end{align}
where $H^{-1}$ denotes the dual space of $H^1$, $c$ and $\Gamma$ are defined in (\ref{mountain-pass}) and (\ref{mountain-pass1}), respectively.
\end{lemma}
\begin{proof}
First, we prove that $\Gamma\neq\emptyset$ and $0<c<\infty$. For any $u\in H^1(\mathbb{R}^3)$, by Lemma \ref{HLS} and Sobolev embedding theorem, we obtain that
\begin{align}\label{Cerami2}
J(u)\geq\frac12\min\{1, \omega\}\|u\|^2_{H^1}-C_1\|u\|^4_{H^1}-C_2\|u\|^p_{H^1}.
\end{align}
It follows from (\ref{Cerami2}) that there exist constants $\rho_0>0$ and $\alpha_0>0$ such that
\begin{align}\label{Cerami3}
J(u)\geq\frac14\min\{1, \omega\}\|u\|^2_{H^1}\geq0, ~~\forall~ \|u\|_{H^1}\leq\rho_0 \quad\hbox{and}\quad J(u)\geq\alpha_0, ~~\forall~ \|u\|_{H^1}=\rho_0.
\end{align}
For any fixed $u\in H^1(\mathbb{R}^3)\setminus\{0\}$, we have
\begin{align*}
J(tu)=\frac{t^2}2\|\nabla u\|^2_{L^2}+\frac{t^2}2 \omega\|u\|^2_{L^2}-\frac{t^4}4\mu\int\phi_u|u|^2dx-\frac{t^p}p\|u\|^p_{L^p}, \quad \forall~ t>0,
\end{align*}
since $p>2$, we obtain that
\begin{align*}
J(tu)\rightarrow-\infty, ~~\hbox{as}~ t\rightarrow\infty \quad\hbox{and}\quad \sup\limits_{t\geq0}J(tu)<\infty.
\end{align*}
Therefore, we can choose $T>0$ to be large enough such that $J(Tu)<0$. Let $\gamma_T(t)=tTu$ for $t\in[0, 1]$, then $\gamma_T\in C([0, 1], H^1(\mathbb{R}^3))$ such that $\gamma_T(0)=0$, $J(\gamma_T(1))<0$ and $\max\limits_{t\in[0, 1]}J(\gamma_T(t))<\infty$. This shows that $\Gamma\neq\emptyset$ and $c<\infty$. For every $\gamma\in\Gamma$, since $\gamma(0)=0$ and $J(\gamma(1))<0$, then it follows from (\ref{Cerami3}) that $\|\gamma(1)\|_{H^1}>\rho_0$. From the continuity of $\gamma(t)$ and the intermediate value theorem, there exists $t_0\in(0, 1)$ such that $\|\gamma(t_0)\|_{H^1}=\rho_0$. Thus, we have
$$\sup\limits_{t\in[0, 1]}J(\gamma(t))\geq J(\gamma{t_0})\geq\alpha_0>0, ~~ \forall~ \gamma\in\Gamma,$$
which yields
\begin{align}\label{Cerami4}
\infty>c=\inf\limits_{\gamma\in\Gamma}\max\limits_{t\in[0, 1]}J(\gamma(t))\geq\alpha_0>0.
\end{align}

Inspired by \cite{Jeanjean1997}, we define the continuous map $h: \mathbb{R}\times H^1(\mathbb{R}^3)\rightarrow H^1(\mathbb{R}^3)$ for $s\in\mathbb{R}$, $v\in H^1(\mathbb{R}^3)$ and $x\in\mathbb{R}^3$ by $h(s, v)(x)=e^sv(e^sx)$, where $\mathbb{R}\times H^1(\mathbb{R}^3)$ is a Banach space equipped with the product norm $\|(s, v)\|_{\mathbb{R}\times H^1(\mathbb{R}^3)}:=\left(|s|^2+\|v\|^2_{H^1}\right)^{\frac12}$. For every $s\in\mathbb{R}$, $v\in H^1(\mathbb{R}^3)$, we consider the following auxiliary functional:
\begin{align}\nonumber
\Psi(s, v):=&J(h(s, v))\\ \nonumber
=&\frac12\|\nabla h(s, v)\|^2_{L^2}+\frac{\omega}2\|h(s, v)\|^2_{L^2}-\frac{\mu}4\int\phi_{h(s, v)}|h(s, v)|^2dx-\frac1p\|h(s, v)\|^p_{L^p}\\ \label{Cerami5}
=&\frac{e^s}2\|\nabla v\|^2_{L^2}+\frac{e^{-s}}2\omega\|v\|^2_{L^2}-\frac{e^{-s}}4\mu\iint\frac{1-e^{-\frac{|x-y|}{ae^s}}}{|x-y|}|v(x)|^2|v(y)|^2dxdy-\frac{e^{(p-3)s}}p\|v\|^p_{L^p}.
\end{align}
It is easy to see that $\Psi\in C^1(\mathbb{R}\times H^1(\mathbb{R}^3), \mathbb{R})$. From (\ref{Pohozaev-equality1}), (\ref{Cerami5}) and the definition of $h$, for $s\in\mathbb{R}$ and $v\in H^1(\mathbb{R}^3)$, we deduce that
\begin{align}\nonumber
\frac{\partial}{\partial s}\Psi(s, v)=&\frac{e^s}2\|\nabla v\|^2_{L^2}-\frac{e^{-s}}2\omega\|v\|^2_{L^2}+\frac{e^{-s}}4\mu\iint\frac{1-e^{-\frac{|x-y|}{ae^s}}}{|x-y|}|v(x)|^2|v(y)|^2dxdy\\ \nonumber
&+\frac{e^{-2s}}{4a}\mu\iint e^{-\frac{|x-y|}{ae^s}}|v(x)|^2|v(y)|^2dxdy-\frac{(p-3)e^{(p-3)s}}p\|v\|^p_{L^p}\\ \label{Cerami6}
=&I(h(s, v)).
\end{align}
Moreover, since the map $v\mapsto h(s, v)$ is linear for fixed $s\in\mathbb{R}$, we have
\begin{align}\label{Cerami7}
\langle\frac{\partial}{\partial v}\Psi(s, v), w\rangle=\langle J'(h(s, v)), h(s, w)\rangle,~~ \hbox{for}~ s\in\mathbb{R} ~ \hbox{and}~ v, w\in H^1(\mathbb{R}^3).
\end{align}

Now, we define a minimax value $\tilde{c}$ for $\Psi$ by
$$\tilde{c}:=\inf\limits_{\tilde{\gamma}\in\widetilde{\Gamma}}\max\limits_{t\in[0, 1]}\Psi(\tilde{\gamma}(t)),$$
where
$$\widetilde{\Gamma}:=\left\{\tilde{\gamma}\in C([0, 1], \mathbb{R}\times H^1(\mathbb{R}^3)):~ \tilde{\gamma}(0)=(0, 0), \Psi(\tilde{\gamma}(1))<0\right\}.$$
Noting that $\Gamma=\left\{h\circ\tilde{\gamma}:~ \tilde{\gamma}\in\widetilde{\Gamma}\right\}$, the minimax value of $J$ and $\Psi$ coincide, i.e. $c=\tilde{c}$. By the definition of $c$, for every $n\in\mathbb{N}$, there exists $\gamma_n\in\Gamma$ such that
$$\max\limits_{t\in[0, 1]}\Psi(0, \gamma_n(t))=\max\limits_{t\in[0, 1]}J(\gamma_n(t))\leq c+\frac1{n^2}.$$
Next, we apply Lemma \ref{minimax-principle} to $\Psi$, $M=[0, 1]$, $M_0=\{0, 1\}$ and $\mathbb{R}\times H^1(\mathbb{R}^3)$, $\widetilde{\Gamma}$ in place of $X$, $\Gamma$. Set $\varepsilon=\varepsilon_n:=\frac1{n^2}$, $\delta=\delta_n:=\frac1n$ and $\tilde{\gamma}_n(t):=(0, \gamma_n(t))$. Since (\ref{Cerami4}) implies that $\varepsilon_n\in(0, \frac{c}2)$ for large $n\in\mathbb{N}$. From Lemma \ref{minimax-principle}, there exists a sequence $\{(s_n, v_n)\}^{\infty}_{n=1}\subset\mathbb{R}\times H^1(\mathbb{R}^3)$ such that, as $n\rightarrow\infty$,
\begin{align}\label{Cerami8}
&\Psi(s_n, v_n)\rightarrow c,\\ \label{Cerami9}
&(1+\|(s_n, v_n)\|_{\mathbb{R}\times H^1(\mathbb{R}^3)})\|\Psi'(s_n, v_n)\|_{(\mathbb{R}\times H^1(\mathbb{R}^3))^{-1}}\rightarrow0,\\ \label{Cerami10}
&\hbox{dist}\left((s_n, v_n), (0, \gamma([0, 1]))\right)\rightarrow0.
\end{align}
Moreover, (\ref{Cerami10}) implies that $s_n\rightarrow0$, as $n\rightarrow\infty$. For every $(\tau, w)\in\mathbb{R}\times H^1(\mathbb{R}^3)$,
\begin{align}\label{Cerami11}
\langle\Psi'(s_n, v_n), (\tau, w)\rangle=\langle J'(h(s_n, v_n)), h(s_n, w)\rangle+I(h(s_n, v_n))\tau.
\end{align}
By (\ref{Cerami6}) and (\ref{Cerami7}), we can take $\tau=1$ and $w=0$ in (\ref{Cerami11}) to obtain
\begin{align}\label{Cerami12}
I(h(s_n, v_n))\rightarrow0 \quad\hbox{as}\quad n\rightarrow\infty.
\end{align}
Let $u_n:=h(s_n, v_n)$. Then it follows from (\ref{Cerami8}) and (\ref{Cerami12}) that
$$J(u_n)\rightarrow c \quad\hbox{and}\quad I(u_n)\rightarrow0, ~~\hbox{as}~~ n\rightarrow\infty.$$
Finally, for given $v\in H^1(\mathbb{R}^3)$, we consider $w_n(x)=e^{-s_n}v(e^{-s_n}x)\in H^1(\mathbb{R}^3)$ and deduce from (\ref{Cerami9}) and (\ref{Cerami11}) with $\tau=0$ that
\begin{align*}
(1+\|u_n\|_{H^1})\big|\langle J'(u_n), v\rangle\big|&=(1+\|u_n\|_{H^1})\big|\langle J'(h(s_n, w_n)), h(s_n, w_n)\rangle\big|\\
&=o_n(1)\|w_n\|_{H^1}\\
&=o_n(1)\left(e^{-\frac{s_n}2}\|\nabla v\|_{L^2}+e^{\frac{e^s}2}\|v\|_{L^2}\right),
\end{align*}
which together with $s_n\rightarrow0$ as $n\rightarrow\infty$, we have
$$(1+\|u_n\|_{H^1})\|J'(u_n)\|_{H^{-1}}\rightarrow0 \quad\hbox{as}\quad n\rightarrow\infty.$$
The proof is thus finished.
\end{proof}

\begin{lemma}\label{proof-theorem1-lemma1}
Let $p=6$ and $\mu>0$ be sufficiently large, then $0<c<\frac13 K^{\frac32}$, where $K$ is the best Sobolev constant for the embedding $D^{1, 2}(\mathbb{R}^3)\hookrightarrow L^6(\mathbb{R}^3)$.
\end{lemma}
\begin{proof}
From Theorem 1.42 in \cite{Willem1996}, we know that $U(x)=\frac{3^{\frac14}}{(1+|x|^2)^{\frac12}}$ is a minimizer for $K$. Let $\psi\in C^{\infty}_c(\mathbb{R}^3, [0, 1])$ be such that
\begin{equation} \nonumber
\left\{
\begin{aligned}
&\psi(x)=
\left\{
\begin{aligned}
&1, \quad\hbox{if}~~ |x|\leq R,\\
&0, \quad\hbox{if}~~ |x|\geq 2R,
\end{aligned}
\right.\\
&|\nabla\psi(x)|\leq C,  \quad\forall~ x\in\mathbb{R}^3.
\end{aligned}
\right.
\end{equation}
For $\varepsilon>0$, we define
\begin{align}\label{proof-theorem1-lemma1-1}
U_{\varepsilon}:=\varepsilon^{-\frac12}U(\frac{x}{\varepsilon}), \quad u_{\varepsilon}:=\psi(x)U_{\varepsilon}(x).
\end{align}
By \cite{Brezis1983} (see also \cite{Willem1996}), we have the following estimates as $\varepsilon\rightarrow0^+$:
\begin{align}\label{proof-theorem1-lemma1-2}
\|\nabla u_{\varepsilon}\|^2_{L^2}=K^{\frac32}+O(\varepsilon), \quad\|u_{\varepsilon}\|^6_{L^6}=K^{\frac32}+O(\varepsilon^3),
\end{align}
and
\begin{equation}\label{proof-theorem1-lemma1-3}
\|u_{\varepsilon}\|^q_{L^q}=\left\{
\begin{aligned}
&O(\varepsilon^{\frac{q}2}), \quad&\hbox{if}~~ q\in[2, 3),\\
&O(\varepsilon^{\frac{q}2}|\hbox{ln}\varepsilon|), \quad&\hbox{if}~~ q=3,\\
&O(\varepsilon^{\frac{6-q}2}), \quad&\hbox{if}~~ q\in(3, 6).
\end{aligned}
\right.
\end{equation}

For any fixed $\varepsilon$ in (\ref{proof-theorem1-lemma1-1}), from $J(tu_{\varepsilon})\rightarrow-\infty$ as $t\rightarrow\infty$, we have that $\max\limits_{t\geq0}J(tu_{\varepsilon})$ is attained at some $t_{\mu}>0$ and $t_{\mu}$ satisfies
\begin{align*}
t_{\mu}\|\nabla u_{\varepsilon}\|^2_{L^2}+t_{\mu}\omega\|u_{\varepsilon}\|^2_{L^2}=\mu t_{\mu}^3\int\phi_{u_{\varepsilon}}|u_{\varepsilon}|^2dx+t_{\mu}^5\|u_{\varepsilon}\|^6_{L^6},
\end{align*}
that is
\begin{align*}
\|\nabla u_{\varepsilon}\|^2_{L^2}+\omega\|u_{\varepsilon}\|^2_{L^2}=\mu t_{\mu}^2\int\phi_{u_{\varepsilon}}|u_{\varepsilon}|^2dx+t_{\mu}^4\|u_{\varepsilon}\|^6_{L^6},
\end{align*}
thanks to $\frac{\partial}{\partial t}J(tu_{\varepsilon})\big|_{t=t_{\mu}}=0$. Therefore, by (\ref{proof-theorem1-lemma1-2}) and (\ref{proof-theorem1-lemma1-3}), we obtain $t_{\mu}\rightarrow0$ as $\mu\rightarrow\infty$. Then,
\begin{align}\label{proof-theorem1-lemma1-4}
\max\limits_{t\geq0}J(tu_{\varepsilon})=&\frac{t_{\mu}^2}2\|\nabla u_{\varepsilon}\|^2_{L^2}+\frac{t_{\mu}^2}2\|u_{\varepsilon}\|^2_{L^2}-\frac{\mu}4t_{\mu}^4\int\phi_{u_{\varepsilon}}|u_{\varepsilon}|^2dx
-\frac{t_{\mu}^6}6\|u_{\varepsilon}\|^6_{L^6}\rightarrow0
\end{align}
as $\mu\rightarrow\infty$. It follows from (\ref{proof-theorem1-lemma1-4}) and the definition of $c$ that
$$0<c\leq\sup\limits_{t\geq0}J(tu_{\varepsilon})<\frac13 K^{\frac32}$$
for $\mu>0$ large enough. This completes the proof.
\end{proof}

\begin{proof}[Proof of Theorem \ref{theorem1}]  The proof is divided into the following steps.\\
\textbf{Step 1.} The existence of nontrivial solutions.

\textbf{Case (1).} $p\in(2, 6)$ and $\mu>0$.

Let $\{u_n\}^{\infty}_{n=1}$ be a Cerami sequence obtained in Lemma \ref{Cerami-sequence1}. We first show that $\{u_n\}^{\infty}_{n=1}$ is bounded in $H^1(\mathbb{R}^3)$. If $p\in(2, 4)$, we get that
\begin{align}\nonumber
pc+o_n(1)&=pJ(u_n)-\langle J'(u_n), u_n\rangle\\ \nonumber
&=\frac{p-2}2\|\nabla u_n\|^2_{L^2}+\frac{p-2}2\omega\|u_n\|^2_{L^2}+\frac{p-4}4\mu\int\phi_{u_n}|u_n|^2dx\\ \label{proof-theorem1-5}
&\geq\frac{p-2}2\min\{1, \omega\}\|u_n\|^2_{H^1}.
\end{align}
If $p\in[4, 6)$, we obtain that
\begin{align}\nonumber
4c+o_n(1)&=4J(u_n)-\langle J'(u_n), u_n\rangle\\ \nonumber
&=\|\nabla u_n\|^2_{L^2}+\omega\|u_n\|^2_{L^2}+\frac{p-4}p\|u_n\|^p_{L^p}\\ \label{proof-theorem1-6}
&\geq\min\{1, \omega\}\|u_n\|^2_{H^1}.
\end{align}
Thus, the sequence $\{u_n\}^{\infty}_{n=1}$ is bounded in $H^1(\mathbb{R}^3)$ for $p\in(2, 6)$.

For any sufficiently large $n$, we have
\begin{align*}
\|\nabla u_n\|^2_{L^2}+\omega\|u_n\|^2_{L^2}&=2J(u_n)+\frac{\mu}2\int\phi_{u_n}|u_n|^2dx+\frac2p\|u_n\|^p_{L^p}\\
&\geq c>0.
\end{align*}
That is,
\begin{align}\label{proof-theorem1-1}
\liminf\limits_{n\rightarrow\infty}\|u_n\|_{H^1}>0.
\end{align}
We then claim that
\begin{align}\label{proof-theorem1-2}
\liminf\limits_{n\rightarrow\infty}\sup\limits_{y\in\mathbb{R}^3}\int_{B_1(y)}|u_n(x)|^2dx>0.
\end{align}
If it is false, by Lemma \ref{vanishing}, after passing to a subsequence, it follows that $u_n\rightarrow0$ in $L^r(\mathbb{R}^3)$ for any $r\in(2, 6)$. Then, Lemma \ref{HLS} gives
\begin{align*}
\int\phi_{u_n}|u_n|^2dx\leq C\|u_n\|^4_{L^{\frac{12}5}}=o_n(1).
\end{align*}
Therefore, we have
$$\min\{1, \omega\}\|u_n\|^2_{H^1}\leq\langle J'(u_n), u_n\rangle+\mu\int\phi_{u_n}|u_n|^2dx+\|u_n\|^p_{L^p}\rightarrow0,$$
as $n\rightarrow\infty$. This contradicts (\ref{proof-theorem1-1}). Hence, (\ref{proof-theorem1-2}) holds. Going if necessary to a subsequence, there exist a sequence $\{y_n\}^{\infty}_{n=1}\subset\mathbb{R}^3$ and a constant $\delta_1>0$ such that $\int_{B_1(y_n)}|u_n|^2dx>\delta_1$. Let $\tilde{u}_n(x):=u_n(x+y_n)$. By Lemma \ref{property2}, we get $\phi_{u_n}(x+y_n)=\phi_{\tilde{u}_n}(x)$. Then (\ref{Cerami1}) gives
\begin{align*}
J(\tilde{u}_n)\rightarrow c, \quad J'(\tilde{u}_n)\rightarrow0, \quad I(\tilde{u}_n)\rightarrow0 \quad\hbox{as}~~ n\rightarrow\infty,
\end{align*}
and $\int_{B_1(0)}|\tilde{u}_n|^2dx>\delta_1$ for all $n\in\mathbb{N}$. Therefore, there exists $u\in H^1(\mathbb{R}^3)\setminus\{0\}$ such that, passing to a subsequence,
\begin{equation} \nonumber
\left\{
\begin{aligned}
&\tilde{u}_n\rightharpoonup u \quad\hbox{in}~~ H^1(\mathbb{R}^3),\\
&\tilde{u}_n\rightarrow u \quad\hbox{in}~~ L^r_{loc}(\mathbb{R}^3), ~\forall~ r\in[1, 6),\\
&\tilde{u}_n\rightarrow u \quad\hbox{a.e. in}~~ \mathbb{R}^3.
\end{aligned}
\right.
\end{equation}
By Lemma \ref{BL} and $\langle J'(\tilde{u}_n(x)), w(x)\rangle=\langle J'(u_n(x)), w(x-y_n)\rangle$, we obtain
\begin{align*}
0=&\lim\limits_{n\rightarrow\infty}\langle J'(\tilde{u}_n), w\rangle\\
=&\lim\limits_{n\rightarrow\infty}\left(\int\nabla\tilde{u}_n\cdot\nabla wdx+\omega\int\tilde{u}_nwdx-\mu\int\phi_{\tilde{u}_n}\tilde{u}_nwdx-\int|\tilde{u}_n|^{p-2}\tilde{u}_nwdx\right)\\
=&\int\nabla u\cdot wdx+\omega\int uwdx-\mu\int\phi_u uwdx-\int|u|^{p-2}uwdx\\
=&\langle J'(u), w\rangle,
\end{align*}
for any $w\in H^1(\mathbb{R}^3)$. Thus, $u$ is a nontrivial solution of Eq.(\ref{main2}).

\textbf{Case (2).} $p=6$ and $\mu>0$ large enough.

Let $\{u_n\}^{\infty}_{n=1}$ be a Cerami sequence obtained in Lemma \ref{Cerami-sequence1}. Then,
\begin{align}\nonumber
c+o_n(1)&=J(u_n)-\frac14\langle J'(u_n), u_n\rangle\\ \nonumber
&\frac14\|\nabla u_n\|^2_{L^2}+\frac{\omega}4\|u_n\|^2_{L^2}+\frac1{12}\|u_n\|^p_{L^p}\\ \label{proof-theorem1-7}
&\geq\frac14\min\{1, \omega\}\|u_n\|^2_{H^1},
\end{align}
which implies that the sequence $\{u_n\}^{\infty}_{n=1}$ is bounded in $H^1(\mathbb{R}^3)$. Thus, there exists $u\in H^1(\mathbb{R}^3)$ such that, up to a subsequence, $u_n\rightharpoonup u$ in $H^1(\mathbb{R}^3)$.

Now, we claim that $u\not\equiv0$. If not, then Lemma \ref{vanishing} means that $u_n\rightarrow0$ in $L^r(\mathbb{R}^3)$ for any $r\in(2, 6)$. Noticing that $\{u_n\}^{\infty}_{n=1}$ is bounded in $H^1(\mathbb{R}^3)$, going to a subsequence, we may assume that
$$\|\nabla u_n\|^2_{L^2}+\omega\|u_n\|^2_{L^2}\rightarrow l\in\mathbb{R}.$$
Moreover, since $\langle J'(u_n), u_n\rangle\rightarrow0$, we obtain
\begin{align*}
\|u_n\|^6_{L^6}=\|\nabla u_n\|^2_{L^2}+\omega\|u_n\|^2_{L^2}-\mu\int\phi_{u_n}|u_n|^2dx\rightarrow l,
\end{align*}
as $n\rightarrow\infty$. By the definition of the best constant $K$ in (\ref{best-constant}), we have
\begin{align}\label{proof-theorem1-11}
K\left(\int|u_n|^6dx\right)^{\frac13}\leq\int|\nabla u_n|^2dx,
\end{align}
which yields $l\geq Kl^{\frac13}$. Hence, we have either $l=0$ or $l\geq K^{\frac32}$. In the case $l\geq K^{\frac32}$, From  $J(u_n)\rightarrow c$ and $\langle J'(u_n), u_n\rangle\rightarrow0$, we know
\begin{align*}
c+o_n(1)&=J(u_n)=J(u_n)-\frac16\langle J'(u_n), u_n\rangle\\
&=\frac13\|\nabla u_n\|^2_{L^2}+\frac{\omega}3\|u_n\|^2_{L^2}-\frac{\mu}{12}\int\phi_{u_n}|u_n|^2dx\\
&=\frac13 l+o_n(1),
\end{align*}
which means $c=\frac13l\geq\frac13 K^{\frac32}$, which contradicts with the fact
$c<\frac13 K^{\frac32}$ obtained from Lemma \ref{proof-theorem1-lemma1}. In the case $l=0$, one has $\|\nabla u_n\|^2_{L^2}+\omega\|u_n\|^2_{L^2}\rightarrow0$, $\int\phi_{u_n}|u_n|^2dx\rightarrow0$ and $\|u_n\|^6_{L^6}\rightarrow0$. Therefore, $J(u_n)\rightarrow0$ as $n\rightarrow\infty$, which is absurd since $J(u_n)\rightarrow c>0$. Thus, $u$ does not vanish identically.

Next, we claim that $\{u_n\}^{\infty}_{n=1}$ has a convergent subsequence, that is, $u_n\rightarrow u$ in $H^1(\mathbb{R}^3)$ as $n\rightarrow\infty$. Define $v_n:=u_n-u$, then we know $v_n\rightharpoonup0$ in $H^1(\mathbb{R}^3)$ and $v_n\rightarrow0$ a.e. in $\mathbb{R}^3$. Moreover, by the Br\'{e}zis-Lieb Lemma in \cite{Brezis-Lieb1983} and Lemma B.2 in \cite{dAvenia2019}, we have
$$\|\nabla u_n\|^2_{L^2}=\|\nabla v_n\|^2_{L^2}+\|\nabla u\|^2_{L^2}+o_n(1),$$
$$\|u_n\|^6_{L^6}=\|v_n\|^6_{L^6}+\|u\|^6_{L^6}+o_n(1)$$
and
\begin{align*}
\int\phi_{u_n}|u_n|^2dx=\int\phi_{v_n}|v_n|^2dx+\int\phi_u|u|^2dx+o_n(1).
\end{align*}
Since $\int\phi_{v_n}|v_n|^2dx\leq\|v_n\|^4_{L^{\frac{12}5}}=o_n(1)$ as $n\rightarrow\infty$ and $J(u)\geq0$. In fact, for any $\varphi\in H^1(\mathbb{R}^3)$, we have $\langle J'(u_n), \varphi\rangle\rightarrow0$. Passing to the limit as $n\rightarrow\infty$, we obtain
$$\int\nabla u\cdot\nabla\varphi dx+\omega\int u\varphi dx-\mu\int\phi_u u\varphi dx-\int|u|^4u\varphi dx=0$$
for any $\varphi\in H^1(\mathbb{R}^3)$. Then, taking $\varphi=u\in H^1(\mathbb{R}^3)$, we get
$$\|\nabla u\|^2_{L^2}+\omega\|u\|^2_{L^2}=\mu\int\phi_u|u|^2dx+\|u\|^6_{L^6},$$
and so,
$$J(u)=\frac{\mu}4\int\phi_u|u|^2dx+\frac13\|u\|^6_{L^6}\geq0.$$
Consequently, we can get that
\begin{align}\nonumber
c+o_n(1)=&J(u_n)=\frac12\|\nabla u_n\|^2_{L^2}+\frac{\omega}2\|u_n\|^2_{L^2}-\frac{\mu}4\int\phi_{u_n}|u_n|^2dx-\frac16\|u_n\|^6_{L^6}\\ \nonumber
=&\frac12\|\nabla v_n\|^2_{L^2}+\frac{\omega}2\|v_n\|^2_{L^2}-\frac{\mu}4\int\phi_{v_n}|v_n|^2dx-\frac16\|v_n\|^6_{L^6}\\
\nonumber
&+\frac12\|\nabla u\|^2_{L^2}+\frac{\omega}2\|u\|^2_{L^2}-\frac{\mu}4\int\phi_{u}|u|^2dx-\frac16\|u\|^6_{L^6}+o_n(1)\\
\nonumber
=&J(u)+\frac12\|\nabla v_n\|^2_{L^2}+\frac{\omega}2\|v_n\|^2_{L^2}-\frac{\mu}4\int\phi_{v_n}|v_n|^2dx-\frac16\|v_n\|^6_{L^6}+o_n(1)\\
\label{proof-theorem1-8}
\geq&\frac12\left(\|\nabla v_n\|^2_{L^2}+\omega\|v_n\|^2_{L^2}\right)-\frac16\|v_n\|^6_{L^6}+o_n(1).
\end{align}
Similarly, since $\langle J'(u), u\rangle=0$, we have
\begin{align}\nonumber
o_n(1)=&\langle J'(u_n), u_n\rangle\\ \nonumber
=&\|\nabla u_n\|^2_{L^2}+\omega\|u_n\|^2_{L^2}-\mu\int\phi_{u_n}|u_n|^2dx-\|u_n\|^6_{L^6}\\
\nonumber
=&\|\nabla v_n\|^2_{L^2}+\omega\|v_n\|^2_{L^2}-\mu\int\phi_{v_n}|v_n|^2dx-\|v_n\|^6_{L^6}\\
\nonumber
&+\|\nabla u\|^2_{L^2}+\omega\|u\|^2_{L^2}-\mu\int\phi_u|u|^2dx-\|u\|^6_{L^6}+o_n(1)\\
\nonumber
=&\langle J'(u), u\rangle+\|\nabla v_n\|^2_{L^2}+\omega\|v_n\|^2_{L^2}-\mu\int\phi_{v_n}|v_n|^2dx-\|v_n\|^6_{L^6}+o_n(1)\\
\label{proof-theorem1-9}
=&\|\nabla v_n\|^2_{L^2}+\omega\|v_n\|^2_{L^2}-\|v_n\|^6_{L^6}+o_n(1).
\end{align}
From (\ref{proof-theorem1-9}), we know there exists a nonnegative constant $d$ such that
\begin{align*}
\|\nabla v_n\|^2_{L^2}+\omega\|v_n\|^2_{L^2}\rightarrow d, \quad \|v_n\|^6_{L^6}\rightarrow d
\end{align*}
as $n\rightarrow\infty$. Thus, from (\ref{proof-theorem1-8}), we obtain
\begin{align}\label{proof-theorem1-10}
c\geq\frac13 d.
\end{align}
It follows from (\ref{proof-theorem1-11}) that $d\geq Kd^{\frac13}$. Therefore, either $d=0$ or $d\geq K^{\frac32}$. If $d\geq K^{\frac32}$, then we obtain from (\ref{proof-theorem1-10}) that $c\geq\frac13d\geq\frac13 K^{\frac32}$, which contradicts with the fact that $c<\frac13K^{\frac32}$. Hence, $d=0$, and
$$\|\nabla v_n\|^2_{L^2}=\|\nabla(u_n-u)\|^2_{L^2}\rightarrow0, \quad \|v_n\|^2_{L^2}=\|u_n-u\|^2_{L^2}\rightarrow0,$$
that is,
$$\|u_n-u\|_{H^1}\rightarrow0$$
as $n\rightarrow\infty$. And so, we have $J$ has a critical value $c\in(0, \frac13K^{\frac32})$ and thus Eq.(\ref{main2}) has a nontrivial solution.

\textbf{Step 2.} $u\in W^{2, s}(\mathbb{R}^3)$ for every $s>1$.

For $u\in H^1(\mathbb{R}^3)$, the Sobolev embedding theorem implies that $u\in L^{s_0}(\mathbb{R}^3)$, where $s_0\in[2, 6]$. Thus, $|u|^{p-2}u\in L^{s_1}(\mathbb{R}^3)$, where $s_1\in\left[\max\{1, \frac2{p-1}\}, \frac6{p-1}\right]$. By Lemma \ref{HLS}, if
\begin{align}\label{proof-theorem1-3}
\frac1{s_2}=\frac2{s_0}-\frac23>0,
\end{align}
then $\phi_u\in L^{s_2}(\mathbb{R}^3)$. In view of the H\"{o}lder inequality, if
\begin{align}\label{proof-theorem1-4}
\frac1{s_3}=\frac1{s_0}+\frac1{s_2}=\frac3{s_0}-\frac23<1,
\end{align}
then we have $\phi_uu\in L^{s_3}(\mathbb{R}^3)$. By (\ref{proof-theorem1-3}) and (\ref{proof-theorem1-4}), we get $s_3\in[\frac65, 3)$. Set $L^{q}(\mathbb{R}^3):=L^{s_1}(\mathbb{R}^3)\cap L^{s_3}(\mathbb{R}^3)$. Since $\left[\max\{1, \frac2{p-1}\}, \frac6{p-1}\right]\cap[\frac65, 3)\neq\emptyset$, we have
$$-\Delta u+\omega u=\mu\phi_uu+|u|^{p-2}u\in L^{q}(\mathbb{R}^3).$$
Hence, $u\in W^{2, q}(\mathbb{R}^3)$ for $q>1$ by using the classical Calder\'{o}n-Zygmund $L^p$ regularity estimates \cite[Chapter 9]{Gilbarg1983}. According to the Sobolev embedding theorem, $u\in L^{t_1}(\mathbb{R}^3)$ for $t_1\in\left[q, \frac{3q}{3-2q}\right]$ when $q<\frac32$, and $u\in L^{t_2}(\mathbb{R}^3)$ for $t_2\in[q, \infty)$ when $q\geq\frac32$.

The rest of the argument is similar to the proof of Theorem 1.1 in \cite{Li2019} (see also Proposition 4.1 in \cite{Moroz2013}). The proof is complete.
\end{proof}

\renewcommand{\theequation}
{\thesection.\arabic{equation}}
\setcounter{equation}{0}
\section{Ground state solutions} \noindent

In this section, we give the proof of the existence and related properties of ground state solutions. First, we prove the existence of the ground state solutions to Eq.(\ref{main2}) and give the minimax property of the ground state energy of $J$. Then, we establish the positivity, radial symmetry, rotational invariance, and exponential decay of the ground state solutions.

By a simple calculation, we have the following lemma.
\begin{lemma}\label{preliminary-lemma-0}
Let $p\in[4, 6]$. Then,
\begin{align*}
f(\tau):=4\tau^p-p\tau^4+p-4\geq0, \quad \forall~ \tau>0.
\end{align*}
\end{lemma}

In order to obtain the ground state solutions of Eq.(\ref{main2}), we first establish some preliminary lemmas.

\begin{lemma}\label{preliminary-lemma-1}
Let $p\in[4, 6]$. Then,
\begin{align}\label{proof-theorem2-1}
J(u)\geq J(\tau u)+\frac{1-\tau^4}4\langle J'(u), u\rangle+\frac{(1-\tau^2)^2}4\left(\|\nabla u\|^2_{L^2}+\omega\|u\|^2_{L^2}\right)
\end{align}
for all $u\in H^1(\mathbb{R}^3)$ and $\tau>0$.
\end{lemma}
\begin{proof}
For all $u\in H^1(\mathbb{R}^3)$ and all $\tau>0$, by Lemma \ref{preliminary-lemma-0}, we have
\begin{align*}
J(u)-J(\tau u)=&\frac{1-\tau^2}2\|\nabla u\|^2_{L^2}+\frac{1-\tau^2}2\omega\|u\|^2_{L^2}-\frac{1-\tau^4}4\mu\int\phi_u|u|^2dx-\frac{1-\tau^p}p\|u\|^p_{L^p}\\
=&\frac{1-\tau^4}4\langle J'(u), u\rangle+\left(\frac{1-\tau^2}2-\frac{1-\tau^4}4\right)\left(\|\nabla u\|^2_{L^2}+\omega\|u\|^2_{L^2}\right)\\
&+\left(\frac{1-\tau^4}4-\frac{1-\tau^p}p\right)\|u\|^p_{L^p}\\
\geq&\frac{1-\tau^4}4\langle J'(u), u\rangle+\frac{(1-\tau^2)^2}4\left(\|\nabla u\|^2_{L^2}+\omega\|u\|^2_{L^2}\right).
\end{align*}
This shows (\ref{proof-theorem2-1}).
\end{proof}

Remark that (\ref{proof-theorem2-1}) with $\tau\rightarrow0$ gives
\begin{align}\label{proof-theorem2-2}
J(u)\geq\frac14\langle J'(u), u\rangle+\frac14\left(\|\nabla u\|^2_{L^2}+\omega\|u\|^2_{L^2}\right), \quad\forall~ u\in H^1(\mathbb{R}^2).
\end{align}

From Lemma \ref{preliminary-lemma-1}, we have the following corollary.
\begin{corollary}\label{preliminary-corollary}
Let $p\in[4, 6]$. Then for all $u\in\mathcal{N}$,
$$J(u)=\max\limits_{\tau>0}J(\tau u).$$
\end{corollary}

\begin{lemma}\label{preliminary-lemma-2}
Let $p\in[4, 6]$. Then for any $u\in H^1(\mathbb{R}^3)\setminus\{0\}$, there exists a unique $\tau_u>0$ such that $\tau_u u\in\mathcal{N}$.
\end{lemma}
\begin{proof}
Let $u\in H^1(\mathbb{R}^3)\setminus\{0\}$ be fixed and define the function
\begin{align}\label{proof-theorem2-6}
g(\tau):=J(\tau u)=\frac{\tau^2}2\|\nabla u\|^2_{L^2}+\frac{\tau^2}2\omega\|u\|^2_{L^2}-\frac{\tau^4}4\mu\int\phi_u|u|^2dx-\frac{\tau^p}p\|u\|^p_{L^p}
\end{align}
on $(0, \infty)$. Direct calculation has
\begin{align*}
g'(\tau)=\tau\|\nabla u\|^2_{L^2}+\tau\omega\|u\|^2_{L^2}-\tau^3\mu\int\phi_u|u|^2dx-\tau^{p-1}\|u\|^p_{L^p}.
\end{align*}
Then, using (\ref{Nehari-functional}) and (\ref{Nehari-manifold}), it is easy to get that
$$g'(\tau)=0\quad\Leftrightarrow\quad\frac1{\tau}\langle J'(\tau u), \tau u\rangle=0\quad\Leftrightarrow\quad\tau u\in\mathcal{N}.$$
Note that $\lim\limits_{\tau\rightarrow0^+}g(\tau)=0$, $g(\tau)>0$ for $\tau>0$ small enough and $g(\tau)<0$ for $\tau$ large enough. Therefore, $\max\limits_{\tau>0}g(\tau)$ is achieved at $\tau=\tau_u>0$ such that $g'(\tau_u)=0$ and $\tau_u u\in\mathcal{N}$.

Next, we claim that $\tau_u$ is unique for any $u\in H^1(\mathbb{R}^3\setminus\{0\})$. Indeed, for any $u\in H^1(\mathbb{R}^3\setminus\{0\})$, let $\tau_1, \tau_2>0$ be such that $g'(\tau_1)=g'(\tau_2)=0$. Then $\langle J'(\tau_1 u), \tau_1u\rangle=\langle J'(\tau_2 u), \tau_2u\rangle=0$. By Lemma \ref{preliminary-lemma-1}, we have
\begin{align}\nonumber
J(\tau_1 u)\geq& J(\tau_2u)+\frac{\tau_1^4-\tau_2^4}{4\tau_1^4}\langle J'(\tau_1u), \tau_1u\rangle+\frac{(\tau_1^2-\tau_2^2)^2}{4\tau_1^2}\left(\|\nabla u\|^2_{L^2}+\omega\|u\|^2_{L^2}\right)\\ \label{proof-theorem2-3}
=&J(\tau_2u)+\frac{(\tau_1^2-\tau_2^2)^2}{4\tau_1^2}\left(\|\nabla u\|^2_{L^2}+\omega\|u\|^2_{L^2}\right),
\end{align}
and
\begin{align}\nonumber
J(\tau_2 u)\geq& J(\tau_1u)+\frac{\tau_2^4-\tau_1^4}{4\tau_2^4}\langle J'(\tau_2u), \tau_2u\rangle+\frac{(\tau_2^2-\tau_1^2)^2}{4\tau_2^2}\left(\|\nabla u\|^2_{L^2}+\omega\|u\|^2_{L^2}\right)\\ \label{proof-theorem2-4}
=&J(\tau_1u)+\frac{(\tau_2^2-\tau_1^2)^2}{4\tau_2^2}\left(\|\nabla u\|^2_{L^2}+\omega\|u\|^2_{L^2}\right).
\end{align}
Then (\ref{proof-theorem2-3}) and (\ref{proof-theorem2-4}) give $\tau_1=\tau_2$. Thus, $t_u>0$ is unique for any $u\in H^1(\mathbb{R}^3)\setminus\{0\}$.
\end{proof}

\begin{lemma}\label{preliminary-lemma-3}
Let $p\in[4, 6]$. Then,
\begin{itemize}
 \item[(1)] there exists $\rho>0$ such that $\|u\|_{H^1}\geq\rho$, $\forall u\in\mathcal{N}$;
 \item[(2)] $\mathcal{N}$ is a natural constraint for the functional $J$, i.e., critical points of $J$ on $\mathcal{N}$ are critical points of $J$ on $H^1(\mathbb{R}^3)$;
 \item[(3)] $c_g=\inf\limits_{u\in\mathcal{N}}J(u)>0$.
\end{itemize}
\end{lemma}
\begin{proof}
(1) Since $\langle J'(u), u\rangle=0$ for any $u\in\mathcal{N}$, by (\ref{Nehari-functional}) and the Sobolev embedding theorem, we have
\begin{align}\label{proof-theorem2-5}
\min\{1, \omega\}\|u\|^2_{H^1}\leq\mu\int\phi_u|u|^2dx+\|u\|^p_{L^p}\leq\max\{C_1, C_2\}\left(\|u\|^4_{H^1}+\|u\|^p_{H^1}\right).
\end{align}
Set $C_0:=\frac{\min\{1, \omega\}}{\max\{C_1, C_2\}}>0$, then (\ref{proof-theorem2-5}) implies that
$$\|u\|_{H^1}\left(\|u\|^{p-2}_{H^1}+\|u\|^2_{H^1}-C_0\right)\geq0.$$
We consider the function: $\zeta(t):=t^{p-2}+t^2-C_0$, where $t\in(0, \infty)$ and $C_0>0$. Simple calculaton shows that $\zeta(t)$ is strictly increasing on $(0, \infty)$. Furthermore, $\zeta(t)>0$ holds when $t$ is sufficiently large. Consequently, there necessarily exists a $t_0>0$ such that $\zeta(t)\geq0$ for all $t\geq t_0$. This implies the existence of a constant $\rho>0$ such that $\|u\|_{H^1}\geq\rho$ for all $u\in\mathcal{N}$.

(2) For each $u\in\mathcal{N}$, we define $\mathcal{A}(u):=\langle J'(u), u\rangle$, by a direct computation,
\begin{align}\nonumber
\langle\mathcal{A}'(u), u\rangle=&2\|\nabla u\|^2_{L^2}+\omega\|u\|^2_{L^2}-4\mu\int\phi_u|u|^2dx-p\|u\|^p_{L^p}\\ \nonumber
=&-2\|\nabla u\|^2_{L^2}-2\omega\|u\|^2_{L^2}+(4-p)\|u\|^p_{L^p}\\ \nonumber
\leq&-2\min\{1, \omega\}\|u\|^2_{H^1}\\ \label{proof-theorem2-15}
\leq&-2\min\{1, \omega\}\rho^2<0.
\end{align}
Then, there exists $\Lambda\in\mathbb{R}$ such that $J'(u)=\Lambda\mathcal{A}'(u)$. Therefore,
$$0=\langle J'(u), u\rangle=\Lambda\langle\mathcal{A}'(u), u\rangle,$$
which implies $\Lambda=0$ by (\ref{proof-theorem2-15}) and then $J'(u)=\Lambda\mathcal{A}'(u)=0$.

(3) Let $\{u_n\}\subset\mathcal{N}$ be such that $c_g=\lim\limits_{n\rightarrow\infty}J(u_n)$. By (\ref{proof-theorem2-2}) and item (1), we obtain
\begin{align*}
c_g+o_n(1)=J(u_n)=J(u_n)-\frac14\langle J'(u_n), u_n\rangle\geq\frac14\{1, \omega\}\rho^2>0,
\end{align*}
which means that $c_g=\inf\limits_{u\in\mathcal{N}}J(u)>0$. This completes the proof.
\end{proof}

Combining Corollary \ref{preliminary-corollary}, Lemmas \ref{preliminary-lemma-2} and \ref{preliminary-lemma-3}, we establish the following minimax property.
\begin{lemma}\label{preliminary-lemma-4}
Let $p\in[4, 6]$. Then
$$c_g=\inf\limits_{u\in\mathcal{N}}J(u)=\inf\limits_{u\in H^1(\mathbb{R}^3)\setminus\{0\}}\max\limits_{\tau>0}J(\tau u)>0.$$
\end{lemma}

By Theorem \ref{theorem1}, we can obtain the following lemma.

\begin{lemma}[Lifting a solution to a path]\label{preliminary-lemma-5}
Let $p\in[4, 6]$ and $u\in H^1(\mathbb{R}^3)\setminus\{0\}$ be a solution of Eq.(\ref{main2}). Then there exists a path $\gamma\in C([0, 1], H^1(\mathbb{R}^3))$ such that
$$\gamma(0)=0,~~ \gamma(\frac12)=u,~~ J(\gamma(1))<0$$
and
$$J(\gamma(t))<J(u),~~ \hbox{for every}~~ t\in[0, 1]\setminus\{\frac12\}.$$
\end{lemma}
\begin{proof}
The proof follows closely the arguments for the local problem developed
by Jeanjean and Tanaka \cite[lemma 2.1]{Jeanjean2003}. We define the path $\hat{\gamma}: [0, \infty)\rightarrow H^1(\mathbb{R}^3)$ by
\begin{equation} \nonumber
\hat{\gamma}(\tau)(x)=\left\{
\begin{aligned}
&\tau u(x), &\tau>0,\\
&0, &\tau=0.
\end{aligned}
\right.
\end{equation}
It is obvious the function $\hat{\gamma}$ is continuous on $(0, \infty)$. By the definiton of $\hat{\gamma}$, we get that
\begin{align*}
\lim\limits_{\tau\rightarrow0^+}\left(\|\nabla\hat{\gamma}(\tau)\|^2_{L^2}+\omega\|\hat{\gamma}(\tau)\|^2_{L^2}\right)
=\lim\limits_{\tau\rightarrow0^+}\tau^2\left(\|\nabla u\|^2_{L^2}+\omega\|u\|^2_{L^2}\right)=0,
\end{align*}
which implies that $\hat{\gamma}$ is continuous at $0$.
Since $u\in H^1(\mathbb{R}^3)\setminus\{0\}$ satisfies Eq.(\ref{main2}), we have $\langle J'(u), u\rangle=0$. Combining this with (\ref{proof-theorem2-6}) and for every $\tau>0$, we obtain
\begin{align*}
J(\hat{\gamma}(\tau))=\left(\frac{\tau^2}2-\frac{\tau^p}p\right)\left(\|\nabla u\|^2_{L^2}+\omega\|u\|^2_{L^2}\right)-\left(\frac{\tau^4}4-\frac{\tau^p}p\right)\|u\|^p_{L^p}.
\end{align*}
By direct calculation, we have
$$\frac{dJ(\hat{\gamma}(\tau))}{d\tau}\Big|_{\tau=1}=0,$$
that is, $\tau=1$ is a critical point of $J(\hat{\gamma}(\tau))$.

In view of Lemma \ref{preliminary-lemma-2}, we know that $J(\hat{\gamma}(\tau))$ has a unique global maximum at $\tau=1$. Then for every $\tau\in[0, \infty)\setminus\{1\}$, $J(\hat{\gamma}(\tau))<J(u)$. Since
$$\lim\limits_{\tau\rightarrow\infty}J(\hat{\gamma}(\tau))=-\infty,$$
the path $\gamma$ can be defined by a suitable change of variable. The proof is thus finished.
\end{proof}

We now have all the tools available to show that the mountain-pass critical level $c$ defined in (\ref{mountain-pass}) coincides with the ground state energy level $c_g$ defined in (\ref{ground-state}), which completes the proof of Theorem \ref{theorem2}.

\begin{proof}[Proof of Theorem \ref{theorem2}]
In view of Lemma \ref{Cerami-sequence1} and the proof of the Step 1 of Theorem \ref{theorem1}, we obtain a Cerami sequence $\{u_n\}^{\infty}_{n=1}$ in $H^1(\mathbb{R}^3)$ at the mountain-pass level $c>0$, which converges weakly to a solution $u\in H^1(\mathbb{R}^3)\setminus\{0\}$ of Eq.(\ref{main2}). Since $\lim\limits_{n\rightarrow\infty}\langle J'(u_n), u_n\rangle=0$, by the weak convergence of the sequence $\{u_n\}^{\infty}_{n=1}$, the weak lower-semicontinuity of
the norm and the Fatou's lemma, we have
\begin{align*}
J(u)&=J(u)-\frac12\langle J'(u), u\rangle\\
&=\frac{\mu}4\int\phi_u|u|^2dx+\left(\frac12-\frac1p\right)\|u\|^p_{L^p}\\
&\leq\liminf\limits_{n\rightarrow\infty}\left[\frac{\mu}4\int\phi_{u_n}|u_n|^2dx+\left(\frac12-\frac1p\right)\|u_n\|^p_{L^p}\right]\\
&=\liminf\limits_{n\rightarrow\infty}\left[J(u_n)-\frac12\langle J'(u_n), u_n\rangle\right]\\
&=c.
\end{align*}
Since $u$ is a nontrivial solution of Eq.(\ref{main2}), we have $J(u)\geq c_g$ by definition of the ground state energy level $c_g$, and hence $c_g\leq c$.

Let $v\in H^1(\mathbb{R}^3)\setminus\{0\}$ be any solution of Eq.(\ref{main2}). It follows Lemma \ref{preliminary-lemma-5} and (\ref{mountain-pass}) that $c\leq J(v)$ and then $c_g\geq c$. We have thus proved that $J(u)=c_g=c$, for $p\in[4, 6]$. Finally, by Lemma \ref{preliminary-lemma-4}, we obtain $J(u)=\inf\limits_{v\in\mathcal{N}}J(v)=\inf\limits_{v\in H^1(\mathbb{R}^3)\setminus\{0\}}\max\limits_{\tau>0}J(\tau v)>0$, and this concludes the proof.
\end{proof}

In order to obtain qualitative properties of the ground state solutions of Eq.(\ref{main2}), first, we study the positivity of the ground state solutions.

\begin{lemma}\label{property-positive}
Let $p\in[4, 6]$ and satisfy the conditions in Theorem \ref{theorem1}. If $u\in H^1(\mathbb{R}^3)$ is a ground state solution of Eq.(\ref{main2}), then $|u|\in\mathcal{N}$, $J(|u|)=c_g$ and $\|\nabla|u|\|_{L^2}=\|\nabla u\|_{L^2}$. Moreover, $|u|>0$ in $\mathbb{R}^3$.
\end{lemma}
\begin{proof}
It follows from $\|\nabla|u| \|^2_{L^2}\leq\|\nabla u\|^2_{L^2}$ that $\langle J'(|u|), |u|\rangle\leq0$. By the continuity of $J(\tau|u|)$, there exists $\tau_{|u|}\in(0, 1]$ satisfying $J'(\tau_{|u|}|u|)=0$, i.e. $\tau_{|u|}|u|\in\mathcal{N}$. Moreover, Lemma \ref{preliminary-lemma-2} guarantees the uniqueness of $\tau_{|u|}$. Thus,
\begin{align*}
J(\tau_{|u|}|u|)&=J(\tau_{|u|}|u|)-\frac12\langle J'(\tau_{|u|}|u|), \tau_{|u|}|u|\rangle\\
&=\frac{\mu}4\tau_{|u|}^4\int\phi_{|u|}|u|^2dx+\frac{p-2}{2p}\tau_{|u|}^p\|u\|^p_{L^p}\\
&\leq\frac{\mu}4\int\phi_u|u|^2dx+\frac{p-2}{2p}\|u\|^p_{L^p}\\
&=J(u)-\frac12\langle J'(u), u\rangle=c_g.
\end{align*}
By the definition of $c_g$, we have $\tau_{|u|}=1$. Therefore, $|u|\in\mathcal{N}$, $J(|u|)=c_g$ and $\|\nabla|u|\|_{L^2}=\|\nabla u\|_{L^2}$. Then, $|u|$ satisfies the equation
\begin{align*}
-\Delta|u|+\omega|u|-\mu\phi_{|u|}|u|=|u|^{p-1}.
\end{align*}
Since $|u|$ is continuous by Theorem \ref{theorem1}, by the strong maximum principle we conclude that $|u|>0$ in $\mathbb{R}^3$.
\end{proof}

Next, we prove that any ground state solutions of Eq.(\ref{main2}) are radial. We follow the arguments of \cite{Moroz2013}, which relies on polarizations. So we first recall some elements of the theory of polarizations (see in \cite{Brock2000, Moroz2015, Schaftingen2008}).

Assume that $H\subset\mathbb{R}^3$ is a closed half-space and that $\sigma_H$ is the reflection with respect to $\partial H$. The polarization $u^H: \mathbb{R}^3\rightarrow\mathbb{R}$ of $u: \mathbb{R}^3\rightarrow\mathbb{R}$ is defined for $x\in\mathbb{R}^3$ by
\begin{equation} \nonumber
u^H(x)=\left\{
\begin{aligned}
&\max\left\{u(x), u\left(\sigma_H(x)\right)\right\}, &\hbox{if}~~ x\in H,\\
&\min\left\{u(x), u\left(\sigma_H(x)\right)\right\}, &\hbox{if}~~ x\notin H.
\end{aligned}
\right.
\end{equation}

We will use the following standard property of polarizations (see in \cite[Lemma 5.3]{Brock2000}).

\begin{lemma}[Polarization and Dirichlet integrals]\label{property-radial-1}
If $u\in H^1(\mathbb{R}^3)$, then $u^H\in H^1(\mathbb{R}^3)$ and
$$\int|\nabla u^H|^2dx=\int|\nabla u|^2dx.$$
\end{lemma}

We shall also use a polarization inequality with equality cases.

\begin{lemma}[Polarization and nonlocal integrals]\label{property-radial-2}
Let $u\in L^{\frac65}(\mathbb{R}^3)$ and $H\subset\mathbb{R}^3$ be a closed half-space. If $u\geq0$, then
\begin{align*}
\iint\frac{1-e^{-\frac{|x-y|}a}}{|x-y|}u(x)u(y)dxdy\leq\iint\frac{1-e^{-\frac{|x-y|}a}}{|x-y|}u^H(x)u^H(y)dxdy,
\end{align*}
with equality if and only if either $u^H=u$ or $u^H=u\circ\sigma_H$.
\end{lemma}
\begin{proof}
The proof of Lemma \ref{property-radial-2} is similar to the proof of \cite[Lemma 5.3]{Moroz2013}, and it will be omitted.
\end{proof}

The last tool that we need is a characterization of symmetric functions by polarizations (see in \cite[Lemma 5.4]{Moroz2013} or \cite[Proposition 3.15]{Schaftingen2008}).

\begin{lemma}[Symmetry and polarization]\label{property-radial-3}
Assume that $u\in L^2(\mathbb{R}^3)$ is nonnegative. There exist $x_0\in\mathbb{R}^3$ and a nonincreasing function $v: (0, \infty)\rightarrow\mathbb{R}$ such that for almost every $x\in\mathbb{R}^3$, $u(x)=v(|x-x_0|)$ if and only if for every closed half-space $H\subset\mathbb{R}^3$, $u^H=u$ or $u^H=u\circ\sigma_H$.
\end{lemma}

Now, we are ready to prove the radial symmetry result.

\begin{lemma}\label{property-radial-4}
Let $p\in[4, 6]$ and satisfy the conditions in Theorem \ref{theorem1}. If $u$ be a positive ground state solution of Eq.(\ref{main2}), then there exists $x_0\in\mathbb{R}^3$ and a nonincreasing positive function $v: (0, \infty)\rightarrow\mathbb{R}$ such that for almost every $x\in\mathbb{R}^3$, $u(x)=v(|x-x_0|)$.
\end{lemma}
\begin{proof}
By Lemma \ref{property-positive}, we can assume that $u>0$. In view of Lemma \ref{preliminary-lemma-5}, there exists a path $\gamma\in C([0, 1], \mathbb{R}^3)$ such that $\gamma(\frac12)=u$ and $\gamma(t)\geq0$ for every $t\in[0, 1]$. For every half-space $H\subset\mathbb{R}^3$, we define the path $\gamma^H: [0, 1]\rightarrow H^1(\mathbb{R}^3)$ by $\gamma^H(t)=(\gamma(t))^H$. From Lemma \ref{property-radial-1} and $\|u^H\|_{L^r}=\|u\|_{L^r}$ with $r\in[1, \infty)$, we have $\gamma^H\in C([0, 1], H^1(\mathbb{R}^3))$. By Lemmas \ref{property-radial-1} and \ref{property-radial-2}, we obtain that $J(\gamma^H(t))\leq J(\gamma(t))$ for every $t\in[0, 1]$ and then $\gamma^H\in\Gamma$. Hence,
$$\max\limits_{t\in[0, 1]}J(\gamma^H(t))\geq c_g.$$
Since for every $t\in[0, 1]\setminus\{\frac12\}$,
$$J(\gamma^H(t))\leq J(\gamma(t))<J(u)=c_g,$$
we deduce that
\begin{align*}
J\left(\gamma^H(\frac12)\right)=J\left((\gamma(\frac12))^H\right)=J(u^H)=c_g.
\end{align*}
Therefore, $J(u^H)=J(u)$, which implies that
\begin{align*}
\int\phi_{u^H}|u^H|^2dx=\int\phi_u|u|^2dx.
\end{align*}
By Lemma \ref{property-radial-2}, we have $u^H=u$ or $u^H=u\circ\sigma_H$. Since this
holds for arbitrary $H$, we conclude by Lemma \ref{property-radial-3} that $u$ is radial and radially
decreasing.
\end{proof}

\begin{proof}[Proof of Theorem \ref{theorem4}]
(1) and (2) are the direct results of Lemmas \ref{property-positive} and \ref{property-radial-4}.

Next, we prove (3). We follow the arguments of Theorem 4.1 in \cite{Hajaiej2004}. Let $u\in\mathcal{N}$ be a ground state solution of Eq.(\ref{main2}), by Lemma \ref{property-positive}, for any $\theta\in\mathbb{R}$, we have
\begin{align*}
\langle J'(e^{i\theta}|u|), e^{i\theta}|u|\rangle=\langle J'(u), u\rangle=0,
\end{align*}
and
\begin{align*}
J(e^{i\theta}|u|)=J(|u|)=c_g.
\end{align*}
Thus, $e^{i\theta}|u|$ is also a ground state solution of Eq.(\ref{main2}). Next, we show $u=e^{i\theta}|u|$ for some $\theta\in\mathbb{R}$. Denote $u(x)=u_1(x)+iu_2(x)$. By Lemma \ref{property-positive}, we know $|u|=(u^2_1+u^2_2)^{\frac12}>0$ and $\|\nabla|u|\|_{L^2}=\|\nabla u\|_{L^2}$. Hence,
\begin{align}\label{proof-theorem2-7}
\int\sum\limits_{j=1}\limits^{3}\frac{(u_1\partial_ju_2-u_2\partial_ju_1)^2}{u_1^2+u_2^2}dx=0.
\end{align}
Since $u$ satisfies Eq.(\ref{main2}), it follows that $u_1$ and $u_2$ satisfy
\begin{equation}\label{proof-theorem2-8}
\left\{
\begin{aligned}
&-\Delta u_1+\omega u_1-\mu\phi_uu_1=|u|^{p-2}u_1, &x\in\mathbb{R}^3,\\
&-\Delta u_2+\omega u_2-\mu\phi_uu_2=|u|^{p-2}u_2, &x\in\mathbb{R}^3.
\end{aligned}
\right.
\end{equation}
Using the elliptic regularity theory, we deduce that $u_1, u_2\in C(\mathbb{R}^3)$. Indeed, since $|u|$ ia a solution to Eq.(\ref{main2}) by Lemma \ref{property-positive}, it follows from the proof of the Step 2 of Theorem \ref{theorem1} that $\phi_u|u|^2\in L^{\infty}(\mathbb{R}^3)$ and $|u|\in W^{2, s}_{loc}(\mathbb{R}^3)$ for any $s>1$. Hence, $|u|\in C(\mathbb{R}^3)$ and then $|u|\in L^{\infty}_{loc}$. From (\ref{proof-theorem2-8}), $u_1, u_2\in H^1(\mathbb{R}^3)$ are solutions to the equation
$$-\Delta w+\omega w=F$$
with
\begin{align*}
F:=\mu\phi_uw+|u|^{p-2}w=\mu\phi_u|u|\frac{w}{|u|}+|u|^{p-1}\frac{w}{|u|}\in L^{\infty}_{loc}(\mathbb{R}^3).
\end{align*}
Thus, $u_1, u_2\in C(\mathbb{R}^3)$ by Theorem 4.13 in \cite{Han1997}.

Let $\Omega:=\{x\in\mathbb{R}^3: u_2(x)=0\}$. The continuity of $u_2$ implies that $\Omega$ is closed. Suppose now that $x_0\in\Omega$, since $|u(x_0)|>0$, there exists an open ball $B$ with centre $x_0$ such that $u_1(x)\neq0$ for all $x\in B$. Therefore, for $x\in B$,
\begin{align*}
\frac{(u_1\partial_ju_2-u_2\partial_ju_1)^2}{u^2_1+u^2_2}=\left[\partial_j\left(\frac{u_2}{u_1}\right)\right]^2\frac{u_1^4}{u_1^2+u_2^2}, \quad j=1, 2, 3,
\end{align*}
which combined with (\ref{proof-theorem2-7}) gives that
\begin{align*}
\int_B\Big|\nabla\left(\frac{u_2}{u_1}\right)\Big|^2\frac{u_1^4}{u_1^2+u_2^2}dx=0.
\end{align*}
Thus, $\nabla\left(\frac{u_2}{u_1}\right)\equiv0$ on $B$ and then there exists a constant $C\in\mathbb{R}^3$ such that $\frac{u_2}{u_1}\equiv C$ on $B$. Since $x_0\in B$ and $u_2(x_0)=0$, we have that $C=0$, this means that $\Omega$ is also an open subset of $\mathbb{R}^3$. Hence, either $u_2\equiv0$ or $u_2\neq0$ for all $x\in\mathbb{R}^3$.

If $u_2\equiv0$ on $\mathbb{R}^3$, then $|u|=|u_1|>0$ on $\mathbb{R}^3$ and so $u(x)=u_1(x)=e^{i\theta}|u|$, where $\theta=0$ if $u_1>0$ and $\theta=\pi$ if $u_1<0$. If $u_2\neq0$ for all $x\in\mathbb{R}^3$. Then,
\begin{align*}
\frac{(u_1\partial_ju_2-u_2\partial_ju_1)^2}{u^2_1+u^2_2}=\left[\partial_j\left(\frac{u_1}{u_2}\right)\right]^2\frac{u_2^4}{u_1^2+u_2^2}, \quad j=1, 2, 3,
\end{align*}
for all $x\in\mathbb{R}^3$ and it follows from (\ref{proof-theorem2-7}) that $\nabla\left(\frac{u_1}{u_2}\right)\equiv0$ on $\mathbb{R}^3$. Thus, there exists a constant $C\in\mathbb{R}$ such that $u_1\equiv Cu_2$ on $\mathbb{R}^3$. Therefore, in complex notation, $u=(C+i)u_2$ and $|u|=|C+i||u_2|$. Let $\beta\in\mathbb{R}$ be such that $C+i=|C+i|e^{i\beta}$ and let $\tilde{\beta}=0$ if $u_2>0$ and $\tilde{\beta}=\pi$ if $u_2<0$ on $\mathbb{R}^3$. Setting $\theta=\beta+\tilde{\beta}$, we have
$$u=(C+i)u_2=|C+i|e^{i\beta}e^{i\tilde{\beta}}|u_2|=e^{i\theta}|u|.$$
This proves (3).

Finally, we show (4). By Lemmas \ref{property-positive} and \ref{property-radial-4}, we know that there exists $x_0\in\mathbb{R}^3$ and a nonincreasing positive function $v: (0, \infty)\rightarrow\mathbb{R}$ such that $|u(x)|=v(|x-x_0|)$ for almost every $x\in\mathbb{R}^3$. Hence, $w:=|u(x+x_0)|$ is a positive and radial nonincreasing solution to Eq.(\ref{main2}), that is, $w$ satisfies the following equation:
\begin{align}\label{proof-theorem2-22}
-\Delta w+\omega w=\mu\phi_ww+w^{p-1}.
\end{align}

In the following, we use the Moser iteration technique (see \cite{Gilbarg1983, Kang2024}). The proof is divided into three steps.\\
\textbf{Step 1.} There exists $C>0$ such that $\|w\|_{L^{\infty}}\leq C$.

Let $A_n=\{x\in\mathbb{R}^3: w(x)\leq n\}$, $B_n=\mathbb{R}^3\setminus A_n$, where $n\in\mathbb{N}$. For any $q>0$, define
\begin{equation} \nonumber
u_n=\left\{
\begin{aligned}
&w^{2q+1}, \quad &x\in A_n,\\
&n^{2q}w, \quad &x\in B_n.
\end{aligned}
\right.
\end{equation}
and
\begin{equation} \nonumber
v_n=\left\{
\begin{aligned}
&w^{q+1}, \quad &x\in A_n,\\
&n^{q}w, \quad &x\in B_n.
\end{aligned}
\right.
\end{equation}
Thus, $u_n, v_n\in H^1(\mathbb{R}^3)$ and $v^2_n=wu_n\leq w^{2(q+1)}$. By simple calculation, we have
\begin{equation} \nonumber
\nabla u_n=\left\{
\begin{aligned}
&(2q+1)w^{2q}\nabla w, \quad &x\in A_n,\\
&n^{2q}\nabla w, \quad &x\in B_n.
\end{aligned}
\right.
\end{equation}
and
\begin{equation} \nonumber
\nabla v_n=\left\{
\begin{aligned}
&(q+1)w^{q}\nabla w, \quad &x\in A_n,\\
&n^{q}\nabla w, \quad &x\in B_n.
\end{aligned}
\right.
\end{equation}
Then, taking $u_n$ as a test function in Eq.(\ref{proof-theorem2-22}), we get
\begin{align*}
\int\nabla w\cdot\nabla u_ndx+\omega\int wu_ndx=\mu\int\phi_w wu_ndx+\int w^{p-1}u_ndx,
\end{align*}
which and $\omega>0$ imply that
\begin{align}\label{proof-theorem2-16}
\int_{A_n}(2q+1)w^{2q}|\nabla w|^2dx+\int_{B_n}n^{2q}|\nabla w|^2dx\leq\mu\int\phi_wwu_ndx+\int w^{p-1}u_ndx.
\end{align}
Since
\begin{align}\label{proof-theorem2-17}
\int|\nabla v_n|^2dx=\int_{A_n}(q+1)^2 w^{2q}|\nabla w|^2dx+\int_{B_n}n^{2q}|\nabla w|^2dx,
\end{align}
then combining (\ref{proof-theorem2-16}) and (\ref{proof-theorem2-17}), and by Step 2 in the proof of Theorem \ref{theorem1} and Lemma \ref{L-infty}, we have
\begin{align*}
\frac{2q+1}{(q+1)^2}\int|\nabla v_n|^2dx&=(2q+1)\int_{A_n}w^{2q}|\nabla w|^2dx+\frac{2q+1}{(q+1)^2}\int_{B_n}n^{2q}|\nabla w|^2dx\\
&\leq\mu\int\phi_wwu_ndx+\int w^{p-1}u_ndx\\
&\leq\int(\mu\phi_w+w^{p-2})v^2_ndx\\
&\leq\varepsilon\int|\nabla v_n|^2dx+C(\varepsilon, w)\int v^2_ndx.
\end{align*}
Therefore, there exists $C$, depending on $\varepsilon, q, \mu, w$, such that
\begin{align}\label{proof-theorem2-18}
\int|\nabla v_n|^2dx\leq C\int v^2_ndx.
\end{align}
Moreover, since $v_n=w^{q+1}$ in $A_n$ and $v^2_n\leq w^{2(q+1)}$ on $\mathbb{R}^3$, it follows from (\ref{proof-theorem2-18}) that
\begin{align*}
\left(\int_{A_n}w^{6(q+1)}dx\right)^{\frac13}&=\left(\int_{A_n}v^6_ndx\right)^{\frac13}\\
&\leq\left(\int v^6_ndx\right)^{\frac13}\\
&\leq K^{-1}\int|\nabla v_n|^2dx\\
&\leq C\int w^{2(q+1)}dx,
\end{align*}
where $K$ is defined in (\ref{best-constant}).
Let $n\rightarrow\infty$, we infer that
\begin{align*}
\left(\int w^{6(q+1)}dx\right)^{\frac13}\leq C\int w^{2(q+1)}dx,
\end{align*}
that is, for any $q\geq2$, it holds that $w\in L^{2(q+1)}(\mathbb{R}^3)$ and
\begin{align}\label{proof-theorem2-19}
\|w\|_{L^{6(q+1)}}\leq C^{\frac1{2(q+1)}}\|w\|_{L^{2(q+1)}}.
\end{align}
In the following, we will use an iteration argument. Let $q_1$ be a positive constant such that $2(q_1+1)=6$. Noting that $w\in L^6(\mathbb{R}^3)$, by (\ref{proof-theorem2-19}), we can obtain $w\in L^{6(q_1+1)}(\mathbb{R}^3)$ and
\begin{align*}
\|w\|_{L^{6(q_1+1)}}\leq C^{\frac1{2(q_1+1)}}\|w\|_{L^{2(q_1+1)}}.
\end{align*}
Choosing $q_2$ satisfying $2(q_2+1)=6(q_1+1)$, we see that $q_2>q_1$ and $w\in L^{2(q_2+1)}(\mathbb{R}^3)$. Thus, by (\ref{proof-theorem2-19}), we get $w\in L^{6(q_2+1)}(\mathbb{R}^3)$ and
\begin{align*}
\|w\|_{L^{6(q_2+1)}}\leq C^{\frac1{2(q_2+1)}}\|w\|_{L^{2(q_2+1)}}.
\end{align*}
Continuing with this iteration, we get a sequence $\{q_n\}^{\infty}_{n=1}$, where $2(q_{n+1}+1)=6(q_n+1)$, such that $w\in L^{6(q_n+1)}(\mathbb{R}^3)$ and
\begin{align*}
\|w\|_{L^{6(q_n+1)}}\leq C^{\frac1{2(q_n+1)}}\|w\|_{L^{2(q_n+1)}}.
\end{align*}
Obviously, we proceed the $k$ times iterations that
\begin{align*}
\|w\|_{L^{6(q_k+1)}}\leq C^{\sum\limits_{j=1}\limits^{k}\frac1{2(q_k+1)}}\|w\|_{L^6}.
\end{align*}
Due to $\lim\limits_{k\rightarrow\infty}C^{\sum\limits_{j=1}\limits^{k}\frac1{2(q_k+1)}}<\infty$, there exists $C>0$, depending on $\varepsilon, q, \mu, K, w$, such that
\begin{align}\label{proof-theorem2-20}
\limsup\limits_{k\rightarrow\infty}\|w\|_{L^{6(q_k+1)}}\leq C.
\end{align}

we now claim that
\begin{align}\label{proof-theorem2-21}
\lim\limits_{k\rightarrow\infty}\|w\|_{L^{6(q_k+1)}}=\|w\|_{L^{\infty}}.
\end{align}
Indeed, if $\|w\|_{L^{\infty}}=\infty$, for any $n\in\mathbb{N}$, there is $\widetilde{\Omega}_n\subset\mathbb{R}^3$ with $|\widetilde{\Omega}_n|<\infty$ such that $w(x)>n$ for all $x\in\widetilde{\Omega}_n$. Hence, $\|w\|_{L^{6(q_k+1)}}\geq n|\widetilde{\Omega}_n|^{\frac1{6(q_k+1)}}$, it means that $\liminf\limits_{k\rightarrow\infty}\|w\|_{L^{6(q_k+1)}}\geq n$. By the arbitrariness of $n$, we have $\liminf\limits_{k\rightarrow\infty}\|w\|_{L^{6(q_k+1)}}=\infty$, which is contradicts with (\ref{proof-theorem2-20}). Thus, $\|w\|_{L^{\infty}}<\infty$. Let $k\geq1$, then $6(q_k+1)>6$ and
\begin{align*}
\left(\int w^{6(q_k+1)}dx\right)^{\frac1{6(q_k+1)}}\leq\left(\|w\|^{6q_k}_{L^{\infty}}\int w^6dx\right)^{\frac1{6(q_k+1)}}
=\|w\|^{1-\frac1{q_k+1}}_{L^{\infty}}\|w\|^{\frac1{q_k+1}}_{L^6},
\end{align*}
which implies that
$$\limsup\limits_{k\rightarrow\infty}\|w\|_{L^{6(q_k+1)}}\leq\|w\|_{L^{\infty}}.$$
Moreover, for any $\varepsilon>0$, taking $\widetilde{\Omega}_{\varepsilon}:=\{x\in\mathbb{R}^3: \|w\|_{L^{\infty}}-\varepsilon\leq w(x)\leq\|w\|_{L^{\infty}}\}$, then $|\widetilde{\Omega}_{\varepsilon}|<\infty$ and
\begin{align*}
\left(\int w^{6(q_k+1)}dx\right)^{\frac1{6(q_k+1)}}\geq(\|w\|_{L^{\infty}}-\varepsilon)
|\widetilde{\Omega}_{\varepsilon}|^{\frac1{6(q_k+1)}},
\end{align*}
which implies that
$$\liminf\limits_{k\rightarrow\infty}\|w\|_{L^{6(q_k+1)}}\geq\|w\|_{L^{\infty}}-\varepsilon.$$
Therefore, we get (\ref{proof-theorem2-21}) by the arbitrariness of $\varepsilon$. In view of (\ref{proof-theorem2-20}) and (\ref{proof-theorem2-21}), we have $\|w\|_{L^{\infty}}\leq C$. \\
\textbf{Step 2.} There exist $r>0$ and $C> 0$ such that for all $|x|\geq r$, we have
\begin{align}\label{proof-theorem2-9}
\int\frac{|w(x-z)|^2}{|z|}dz\leq C|x|^{-\frac12}.
\end{align}

In fact, for $|x|\geq r$, we divide the integral in (\ref{proof-theorem2-9}) into three regions, i.e. $\mathbb{R}^3=\Omega_1\cup\Omega_2\cup\Omega_3$, where $\Omega_1:=\left\{z: z\in\mathbb{R}^3\setminus B_{4|x|}(x)\right\}$, $\Omega_2:=\left\{z: z\in B_{4|x|}(x)\setminus B_{\frac{|x|}2}(x)\right\}$ and $\Omega_3:=\left\{z: z\in B_{\frac{|x|}2}(x)\right\}$. Then,
\begin{align}\nonumber
\int_{\mathbb{R}^3}\frac{|w(x-z)|^2}{|z|}dz&=\left(\int_{\Omega_1}+\int_{\Omega_2}+\int_{\Omega_3}\right)\frac{|w(x-z)|^2}{|z|}dz\\ \label{proof-theorem2-10}
&:=I_1+I_2+I_3.
\end{align}
For $|x|\geq r$ and $z\in\Omega_1$, we have $|z|-|x|\geq|x|$ and $|z|\geq\frac12(|x|+|z|)\geq\frac12|x-z|$. It follows from the proof of the Step 1 of Theorem \ref{theorem1} that
\begin{align}\label{proof-theorem2-11}
\|w(x-z)\|^2_{H^1}=\|w||^2_{H^1}\leq Cc.
\end{align}
By using Lemma \ref{radial-inequality} with $s=2$, the Sobolev embedding theorem and (\ref{proof-theorem2-11}), for any $\delta\in(0, 1)$, we obtain
\begin{align}\nonumber
I_1&\leq C_1\int_{\Omega_1}\frac1{|x|^{\delta}|x-z|^{1-\delta}}|x-z|^{-3}\|w(x-z)\|^2_{L^2}dz\\ \nonumber
&\leq C_1|x|^{-\delta}\int_{4|x|}^{\infty}r^{\delta-4}r^2dr\\ \label{proof-theorem2-12}
&\leq C_1|x|^{-1}.
\end{align}
For $|x|\geq r$ and $z\in\Omega_2$, we have $|z|\leq4|x|$ and $|z-x|\geq\frac{|x|}2$. By using Lemma \ref{radial-inequality} with $s=2$, the Sobolev embedding theorem and (\ref{proof-theorem2-11}), we obtain
\begin{align}\nonumber
I_2&\leq C_2\int_{\Omega_2}\frac1{|z|}\left(\frac{|x|}2\right)^{-3}\|w(x-z)\|^2_{L^2}dz\\ \nonumber
&\leq C_2|x|^{-3}\int_{0}^{4|x|}\frac1{r}r^2dr\\ \label{proof-theorem2-13}
&\leq C_2|x|^{-1}.
\end{align}
For $|x|\geq r$ and $z\in\Omega_3$, we have $|z-x|\leq\frac{|x|}2$ and $|z|\geq\frac{|x|}2$.
By using Lemma \ref{radial-inequality} with $s=\frac{12}5$, the Sobolev embedding theorem and (\ref{proof-theorem2-11}), we obtain
\begin{align}\nonumber
I_3&\leq C_3\int_{\Omega_3}\frac2{|x|}|z-x|^{-\frac52}\|w(x-z)\|^2_{L^\frac{12}5}dz\\ \nonumber
&\leq C_3|x|^{-1}\int_{0}^{\frac{|x|}2}r^{-\frac52}r^2dr\\ \label{proof-theorem2-14}
&\leq C_3|x|^{-\frac12}.
\end{align}
Substituting (\ref{proof-theorem2-12}), (\ref{proof-theorem2-13}) and (\ref{proof-theorem2-14}) into (\ref{proof-theorem2-10}), we obtain that there exists $C>0$ such that (\ref{proof-theorem2-9}) holds for all $|x|\geq r$. \\
\textbf{Step 3.} There exists a constant $M>0$ such that
$$w(x)\leq Me^{-\frac{\sqrt{\omega}}2|x|}, \quad \forall~ x\in\mathbb{R}^3.$$

Let $\vartheta\in C^{\infty}(\mathbb{R}^3, [0, 1])$ such that $|\nabla\vartheta|\leq\frac2{R-r}$ and
\begin{equation} \nonumber
\vartheta=\left\{
\begin{aligned}
&1, \quad |x|\geq R,\\
&0, \quad |x|\leq r.
\end{aligned}
\right.
\end{equation}
Here $R>1$ and $r\in(0, R)$. For all $\iota\geq2$, multiplying both sides of Eq.(\ref{proof-theorem2-22}) by $\varphi:=\vartheta^2 w^{2\iota+1}\in H^1(\mathbb{R}^3)$ and integrate on $\mathbb{R}^3$, we get
\begin{align}\label{proof-theorem2-23}
\int\nabla w\cdot\nabla\varphi dx+\omega\int w\varphi dx=\mu\int\phi_ww\varphi dx+\int w^{p-1}\varphi dx.
\end{align}
Since $\omega>0$ and $\nabla\varphi=(2\iota+1)\vartheta^2 w^{2\iota}\nabla w+2\vartheta w^{2\iota+1}\nabla\vartheta$, then from (\ref{proof-theorem2-23}) we have
\begin{align}\nonumber
\int(2\iota+1)\vartheta^2 w^{2\iota}|\nabla w|^2dx&\leq\bigg|\int 2\vartheta w^{2\iota+1}\nabla\vartheta\cdot\nabla wdx\bigg|+\mu\int\phi_ww\varphi dx+\int w^{p-1}\varphi dx\\
\label{proof-theorem2-24}
&:=L_1+L_2+L_3.
\end{align}
For $L_1$, from Young inequality, we obtain
\begin{align*}
L_1&=\bigg|\int 2\vartheta w^{2\iota+1}\nabla\vartheta\cdot\nabla wdx\bigg|\\
&=\bigg|\int \vartheta w^{\iota}\nabla w\cdot2w^{\iota+1}\nabla\vartheta dx\bigg|\\
&\leq\frac12\int\vartheta^2 w^{2\iota}|\nabla w|^2dx+\int w^{2(\iota+1)}|\nabla\vartheta|^2dx\\
&\leq\frac12\int\vartheta^2 w^{2\iota}|\nabla w|^2dx+\frac4{(R-r)^2}\int_{|x|\geq r}w^{2(\iota+1)}dx.
\end{align*}
For $L_2$, by Step 2, we have
\begin{align*}
\lim\limits_{|x|\rightarrow\infty}\phi_w(x)\leq\lim\limits_{|x|\rightarrow\infty}\int\frac{|w(x-z)|^2}{|z|}dz=0.
\end{align*}
Thus,
\begin{align*}
L_2=\mu\int\phi_ww\varphi dx\leq C\int_{|x|\geq r}w^{2(\iota+1)}dx.
\end{align*}
For $L_3$, by Step 1, we get
\begin{align*}
L_3=\int w^{p-1}\varphi dx\leq C\int_{|x|\geq r}w^{2(\iota+1)}dx.
\end{align*}
Then, from (\ref{proof-theorem2-24}), there exists a positive constant $C>0$ such that
\begin{align}\label{proof-theorem2-25}
\left(2\iota+\frac12\right)\int\vartheta^2 w^{2\iota}|\nabla w|^2dx\leq C\left[1+\frac4{(R-r)^2}\right]\int_{|x|\geq r}w^{2(\iota+1)}dx.
\end{align}
Moreover, for all $\iota>0$, by using the Sobolev embedding theorem and (\ref{proof-theorem2-25}), we have
\begin{align*}
\left(\int_{|x|\geq R}(w^{\iota+1})^6dx\right)^{\frac13}&\leq\left(\int(\vartheta w^{\iota+1})^6dx\right)^{\frac13}\\
&\leq K^{-1}\int\big|\nabla(\vartheta w^{\iota+1})\big|^2dx\\
&=K^{-1}\int\big|w^{\iota+1}\nabla\vartheta+(\iota+1)\vartheta w^{\iota}\nabla w\big|^2dx\\
&\leq C\int\left(w^{2(\iota+1)}|\nabla\vartheta|^2+(\iota+1)^2\vartheta^2 w^{2\iota}|\nabla w|^2\right)dx\\
&\leq C\frac4{(R-r)^2}\int_{|x|\geq r}w^{2(\iota+1)}dx+C(\iota+1)\left[1+\frac4{(R-r)^2}\right]\int_{|x|\geq r}w^{2(\iota+1)}dx\\
&\leq C(\iota+1)\left[1+\frac4{(R-r)^2}\right]\int_{|x|\geq r}w^{2(\iota+1)}dx,
\end{align*}
which implies that
\begin{align}\label{proof-theorem2-26}
\|w\|_{L^{6(\iota+1)}(|x|\geq R)}\leq[C(\iota+1)]^{\frac1{2(\iota+1)}}\left[1+\frac4{(R-r)^2}\right]^{\frac1{2(\iota+1)}}
\|w\|_{L^{2(\iota+1)}(|x|\geq r)}.
\end{align}
In order to use the Moser iteration, let $\iota_1=2$, $\iota_{n+1}+1=3(\iota_n+1)$, $r_n=R-\frac{R}{2^{n+1}}$, then, from (\ref{proof-theorem2-26}) we get
\begin{align}\nonumber
\|w\|_{L^{6(\iota_n+1)}(|x|\geq R)}\leq&\|w\|_{L^{6(\iota_n+1)}(|x|\geq r_n)}\\ \nonumber
\leq&[C(\iota_n+1)]^{\frac1{2(\iota_n+1)}}\left[1+\frac4{(r_n-r_{n-1})^2}\right]^{\frac1{2(\iota_n+1)}}
\|w\|_{L^{2(\iota_n+1)}(|x|\geq r_{n-1})}\\ \nonumber
\leq&[C(\iota_n+1)]^{\frac1{2(\iota_n+1)}}\left(2\times4^{n+2}\right)^{\frac1{2(\iota_n+1)}}
\|w\|_{L^{6(\iota_{n-1}+1)}(|x|\geq r_{n-1})}\\ \nonumber
\leq&[C(\iota_n+1)]^{\frac1{2(\iota_n+1)}}\left(2\times4^{n+2}\right)^{\frac1{2(\iota_n+1)}}
[C(\iota_{n-1}+1)]^{\frac1{2(\iota_{n-1}+1)}}\\ \nonumber
&\left(2\times4^{n+1}\right)^{\frac1{2(\iota_{n-1}+1)}}\|w\|_{L^{2(\iota_{n-1}+1)}(|x|\geq r_{n-2})}\\ \nonumber
&\cdots\\
\label{proof-theorem2-27}
\leq&(2C)^{\sum\limits_{j=1}\limits^{n}\frac1{2(\iota_j+1)}}2^{\sum\limits_{j=1}\limits^{n}
\frac{j+2}{\iota_j+1}}\prod\limits_{j=1}\limits^{n}(\iota_j+1)^{\frac1{2(\iota_j+1)}}
\|w\|_{L^6(|x|\geq r_0)}.
\end{align}
Since $\iota_1=2$, $\iota_{n+1}+1=3(\iota_n+1)$, we can see that $\iota_n+1=3^n$ for all $n\in\mathbb{N}$. Then,
\begin{align*}
\lim\limits_{n\rightarrow\infty}(2C)^{\sum\limits_{j=1}\limits^{n}\frac1{2(\iota_j+1)}}
=\lim\limits_{n\rightarrow\infty}(2C)^{\frac12\sum\limits_{j=1}\limits^{n}\frac1{3^j}}<\infty,
\end{align*}
\begin{align*}
\lim\limits_{n\rightarrow\infty}\prod\limits_{j=1}\limits^{n}(\iota_j+1)^{\frac1{2(\iota_j+1)}}
=\lim\limits_{n\rightarrow\infty}\prod\limits_{j=1}\limits^{n}(3^j)^{\frac1{2\times 3^j}}=\lim\limits_{n\rightarrow\infty}3^{\frac12\sum\limits_{j=1}\limits^{n}\frac{j}{3^j}}<\infty,
\end{align*}
and
\begin{align*}
\lim\limits_{n\rightarrow\infty}2^{\sum\limits_{j=1}\limits^{n}\frac{j+2}{\iota_j+1}}
=\lim\limits_{n\rightarrow\infty}2^{\sum\limits_{j=1}\limits^{n}\frac{j+2}{3^j}}<\infty.
\end{align*}
Therefore, letting $n\rightarrow\infty$ in (\ref{proof-theorem2-27}), similar to (\ref{proof-theorem2-21}), we can conclude that
$$\|w\|_{L^{\infty}(|x|\geq R)}\leq C\|w\|_{L^6(|x|\geq\frac{R}2)}.$$
Thus, for any $\varepsilon>0$ fixed, choosing $R>1$ large enough one infers $\|w\|_{L^{\infty}(|x|\geq R)}\leq\varepsilon$. This shows that
\begin{align}\label{proof-theorem2-28}
w(x)\rightarrow0 \quad\hbox{as}~ |x|\rightarrow\infty.
\end{align}

Now, for any $\mu>0$, by (\ref{proof-theorem2-28}) and the fact that $\omega>0$, there exists $\widetilde{R}>0$ large enough such that
\begin{align*}
-\Delta w=(-\omega+\mu\phi_w+w^{p-2})w\leq-\frac{\omega}4w, \quad\hbox{for all}~ |x|>\widetilde{R}.
\end{align*}
Let $\zeta(x)=M_1e^{-\frac{\sqrt{\omega}}2|x|}$, where $M_1>0$ satisfies
\begin{align*}
M_1e^{-\frac{\sqrt{\omega}}2\widetilde{R}}\geq w(x), \quad\hbox{for}~ |x|=\widetilde{R}.
\end{align*}
By simple calculation, we have $\Delta\zeta\leq\frac{\omega}4\zeta$ for all $x\neq0$. Set $\nu:=\zeta-w$, we get
\begin{equation} \nonumber
\left\{
\begin{aligned}
&-\Delta\nu+\frac{\omega}4\nu\geq0, \quad &|x|>\widetilde{R},\\
&\nu(x)\geq0, \quad &|x|=\widetilde{R},\\
&\lim\limits_{|x|\rightarrow\infty}\nu(x)=0.
\end{aligned}
\right.
\end{equation}
Thus, it follows from maximum principle that $\nu\geq0$, that is,
\begin{align*}
w(x)\leq M_1e^{-\frac{\sqrt{\omega}}2|x|}, \quad\hbox{for all}~ |x|\geq\widetilde{R}.
\end{align*}
Moreover, since $w$ is continuous function, there exists $M_2>0$ such that
\begin{align*}
w(x)e^{\frac{\sqrt{\omega}}2|x|}\leq M_2, \quad\hbox{for all}~ |x|\leq\widetilde{R}.
\end{align*}
Hence, letting $M:=\max\{M_1, M_2\}>0$, we obtain
\begin{align*}
w(x)\leq Me^{-\frac{\sqrt{\omega}}2|x|}, \quad\hbox{for all}~ x\in\mathbb{R}^3.
\end{align*}
The proof is complete.
\end{proof}

\renewcommand{\theequation}
{\thesection.\arabic{equation}}
\setcounter{equation}{0}
\section{Asymptotic behavior} \noindent

In this section, we consider the asymptotic behavior of the solutions to Eq.(\ref{main2}) with respect to $a>0$ and prove Theorem \ref{theorem3}.

\begin{proof}[Proof of Theorem \ref{theorem3}]
Let $\{u^a, \phi^a\}\subset H^1_r(\mathbb{R}^3)\times\mathcal{D}_r$ be the family of the solutions of Eq.(\ref{main2}), where we are using the notation
$$\phi^a=\phi^a_{u^a}=\mathcal{K}\ast|u^a|^2=\int\frac{1-e^{-\frac{|x-y|}a}}{|x-y|}|u^a(y)|^2dy.$$
In contrast to the previous sections we use the explicit dependence on $a$ also in the functional. Then, for each $a>0$, the functions $\{u^a\}$ are critical point of $J$ at the mountain-pass value $c>0$. Thus,
$$\langle J'(u^a), u^a\rangle=0, \quad c=J(u^a).$$
In view of (\ref{proof-theorem1-5}) and (\ref{proof-theorem1-6}), we have that $\{u^a\}$ is bounded in $H^1_r(\mathbb{R}^3)$. Therefore, there exists $u^0\in H^1_r(\mathbb{R}^3)$ such that, up to subsequences,
\begin{equation} \nonumber
\left\{
\begin{aligned}
&u^a\rightharpoonup u^0 \quad\hbox{in}~~ H_r^1(\mathbb{R}^3),\\
&u^a\rightarrow u^0 \quad\hbox{in}~~ L^q(\mathbb{R}^3), ~\forall~ q\in(2, 6),\\
&u^a\rightarrow u^0 \quad\hbox{a.e. in}~~ \mathbb{R}^3,
\end{aligned}
\right.
\end{equation}
as $a\rightarrow0$.
In particular
$$(u^a)^2\rightarrow(u^0)^2 ~~\hbox{in}~~ L^{\frac65}(\mathbb{R}^3), ~~ \hbox{as}~~a\rightarrow0.$$
By (2) of Lemma \ref{asymptotic-lemma}, we have $\phi^a\rightarrow\phi^0$ in $D_r^{1, 2}(\mathbb{R}^3)$ and $a\Delta\phi^a\rightarrow0$ in $L_r^2(\mathbb{R}^3)$ as $a\rightarrow0$, where $\phi^0\in D^{1, 2}(\mathbb{R}^3)$ is the unique solution of $-\Delta\phi=4\pi|u^0|^2$
in $\mathbb{R}^3$.

For any $\varphi\in C^{\infty}_c(\mathbb{R}^3)$, we can have
\begin{align}\label{proof-theorem3-1}
\int\nabla u^a\cdot\nabla\varphi dx+\omega\int u^a\varphi dx-\mu\int\phi^au^a\varphi dx=\int|u^a|^{p-2}u^a\varphi dx.
\end{align}
By standard arguments, as $a\rightarrow0$, we obtain
\begin{align}\label{proof-theorem3-2}
\int\nabla u^a\cdot\nabla\varphi dx\rightarrow\int\nabla u^0\cdot\nabla\varphi dx,
\end{align}
\begin{align}\label{proof-theorem3-3}
\int u^a\varphi dx\rightarrow\int u^0\varphi dx,
\end{align}
\begin{align}\label{proof-theorem3-4}
\int|u^a|^{p-2}u^a\varphi dx\rightarrow\int|u^0|^{p-2}u^0\varphi dx.
\end{align}
Moreover, since $\phi^a\rightarrow\phi^0$ in $L^6(\mathbb{R}^3)$, $u^a\rightarrow u^0$ in $L^{\frac{12}5}(\mathbb{R}^3)$ as $a\rightarrow0$. Using the H\"{o}lder inequality, we get
\begin{align}\label{proof-theorem3-5}
\int\phi^a u^a\varphi dx\rightarrow\int\phi^0 u^0\varphi dx.
\end{align}
Then, by (\ref{proof-theorem3-2})-(\ref{proof-theorem3-5}), we have
\begin{align}\label{proof-theorem3-6}
\int\nabla u^0\cdot\nabla\varphi dx+\omega\int u^0\varphi dx-\mu\int\varphi^0u^0\varphi dx=\int|u^0|^{p-2}u^0\varphi dx,
\end{align}
which shows that $(u^0, \phi^0)$ is a solution of Eq.(\ref{main3}).

In the sequel, we prove that $u^a\rightarrow u^0$ in $H^1_r(\mathbb{R}^3)$ as $a\rightarrow0$. Taking $\varphi=u^a$ and $\varphi=u^0$ respectively in (\ref{proof-theorem3-1}), we obtain
\begin{align}\nonumber
\int&\nabla u^a\cdot\nabla(u^a-u^0)dx+\omega\int u^a(u^a-u^0)dx-\mu\int\phi^au^a(u^a-u^0)dx\\ \label{proof-theorem3-7}
&=\int|u^a|^{p-2}u^a(u^a-u^0)dx.
\end{align}
Similarly, taking $\varphi=u^a$ and $\varphi=u^0$ respectively in (\ref{proof-theorem3-6}), then
\begin{align}\nonumber
\int&\nabla u^0\cdot\nabla(u^a-u^0)dx+\omega\int u^0(u^a-u^0)dx-\mu\int\phi^0u^0(u^a-u^0)dx\\ \label{proof-theorem3-8}
&=\int|u^0|^{p-2}u^0(u^a-u^0)dx.
\end{align}
By the H\"{o}lder inequality, we have
\begin{align}\label{proof-theorem3-9}
\int\phi^au^a(u^a-u^0)dx\leq\|\phi^a\|_{L^6}\|u^a\|_{L^{\frac{12}5}}\|u^a-u^0\|_{L^{\frac{12}5}}\rightarrow0
\end{align}
and
\begin{align}\label{proof-theorem3-10}
\int|u^a|^{p-2}u^a(u^a-u^0)dx\leq\|u^a\|^{p-2}_{L^p}\|u^a\|_{L^p}\|u^a-u^0\|_{L^p}\rightarrow0,
\end{align}
as $a\rightarrow0$. Similarly, we easily get
\begin{align}\label{proof-theorem3-11}
\int\phi^0u^0(u^a-u^0)dx\rightarrow0 \quad\hbox{and}\quad
\int|u^0|^{p-2}u^0(u^a-u^0)dx\rightarrow0, ~~\hbox{as}~~ a\rightarrow0.
\end{align}
Thus, from (\ref{proof-theorem3-7})-(\ref{proof-theorem3-11}), we have
\begin{align*}
\int|\nabla(u^a-u^0)|^2dx+\omega\int|u^a-u^0|^2dx\rightarrow0, ~~\hbox{as}~~ a\rightarrow0.
\end{align*}
It follows from $\omega>0$ that $u^a\rightarrow u^0$ in $H^1_r(\mathbb{R}^3)$. The proof of Theorem \ref{theorem3} is finished.
\end{proof}

 \par


\begin{thebibliography}{}\setlength{\itemsep}{0mm}
\bibitem{Aberqi2022}
Aberqi, A., Bennouna, J., Benslimane, O., Ragusa, M.A.: On $p(z)$-Laplacian system involving critical nonlinearities. J. Funct. Space. 6685771 (2022).

\bibitem{Afonso2023}
Afonso, D., Siciliano, G.: Normalized solutions to a Schr\"{o}dinger-Bopp-Podolsky system under neumann boundary conditions. Commun. Contemp. Math. 25(02), 2150100 (2023).

\bibitem{Azorero1991}
Azorero, J.G., Alonso, I.P.: Multiplicity of solutions for elliptic problems with critical exponent or with a nonsymmetric term. Trans. Amer. Math. Soc. 2, 877-895 (1991).

\bibitem{Berestycki1983}
Berestycki, H., Lions, P.L.: Nonlinear scalar field equations. I. Existence of a ground state. Arch. Ration. Mech. Anal. 82(4), 313-345 (1983).

\bibitem{Bertin2017}
Bertin, M.C., Pimentel, B.M., Valc\'{a}rcel, C.E., Zambrano, G.E.R.: Hamilton-Jacobi formalism for Podolsky's electromagnetic theory on the null-plane. J. Math. Phys. 58(8), 082902 (2017).

\bibitem{Bopp1940}
Bopp, F.: Eine lineare Theorie des Elektrons. Ann. Phys. 38(5), 345-384 (1940).

\bibitem{Born1933}
Born, M.: Modified field equations with a finite radius of the electron. Nature. 132, 282 (1933).

\bibitem{Born1934}
Born, M.: On the quantum theory of the electromagnetic field. Proc. R. Soc. London, Ser. A. 143, 410-437 (1934).

\bibitem{BornI1933}
Born, M., Infeld, L.: Foundations of the new field theory. Nature. 132, 1004 (1933).

\bibitem{BornI1934}
Born, M., Infeld, L.: Foundations of the new field theory. Proc. R. Soc. London, Ser. A. 144, 425-451 (1934).

\bibitem{Brezis-Kato1979}
Br\'{e}zis, H., Kato, T.: Remarks on the Schr\"{o}dinger operator with singularly complex potentials. J. Math. Pures Appl. 58, 137-151 (1979).

\bibitem{Brezis-Lieb1983}
Br\'{e}zis, H., Lieb, E.: A relation between pointwise convergence of functions and convergence of functionals. Proc. Amer. Math. Soc. 88, 486-490 (1983).

\bibitem{Brezis1983}
Br\'{e}zis, H., Nirenberg, L.: Positive solutions of nonlinear elliptic equations involving critical Sobolev exponents. Commun. Pure Appl. Math. 36, 437-477 (1983).

\bibitem{Brock2000}
Brock, F., Solynin, A.Yu.: An approach to symmetrization via polarization. Trans. Amer. Math. Soc. 352(4), 1759-1796 (2000).

\bibitem{Chen2020}
Chen, S., Tang, X.: On the critical Schr\"{o}dinger-Bopp-Podolsky system with general nonlinearities. Nonlinear Anal. 195, 111734 (2020).

\bibitem{Damian2024}
Damian, H. M. S., Siciliano. G.: Critical Schr\"{o}dinger-Bopp-Podolsky systems: solutions in the semiclassical limit. Calc. Var. Partial Differ. Equ. 63(6), 155 (2024).

\bibitem{dAvenia2019}
d'Avenia, P., Siciliano, G.: Nonlinear Schr\"{o}dinger equation in the Bopp-Podolsky electrodynamics: Solutions in the electrostatic case. J. Differ. Equ. 267, 1025-1065 (2019).

\bibitem{Farid2023}
Farid, M., Ali, R., Kazmi, K.R.: Inertial iterative method for a generalized mixed equilibrium, variational inequality and a fixed point problems for a family of quasi-nonexpansive mappings. Filomat. 37(18), 6133-6150 (2023).

\bibitem{Frenkel1996}
Frenkel, F.: $4/3$ problem in classical electrodynamics. Phys. Rev. E. 54, 5859-5862 (1996).

\bibitem{Gilbarg1983}
Gilbarg, D., Trudinger, N.S.: Elliptic Partial Differential Equations of Second Order, 2nd edn. Springer, Berlin (1983).

\bibitem{Hajaiej2004}
Hajaiej, H., Stuart, C.A.: On the variational approach to the stability of standing waves for the nonlinear Schr\"{o}dinger equation. Adv. Nonlinear Stud. 4(4), 469-501 (2004).

\bibitem{Han1997}
Han, Q., Lin, F.H.: Elliptic partial differential equations. New York: New York University, Courant Institute of Mathematical Sciences and American Mathematical Society (1997).

\bibitem{Huang2025}
Huang, J., Wang, S.: Normalized ground states for the mass supercritical Schr\"{o}dinger-Bopp-Podolsky system: Existence, uniqueness, limit behavior, strong instability. J. Differ. Equ. 437, 113282 (2025).

\bibitem{Jeanjean1997}
Jeanjean, L.: Existence of solutions with prescribed norm for semilinear elliptic equations. Nonlinear Anal. 28, 1633-1659 (1997).

\bibitem{Jeanjean2019}
Jeanjean, L., Lu, S.S.: Nonradial normalized solutions for nonlinear scalar field equations. Nonlinearity 32, 4942-4966 (2019).

\bibitem{Jeanjean2003}
Jeanjean, L., Tanaka, K.: A remark on least energy solutions in $\mathbb{R}^N$. Proc. Amer. Math. Soc. 131(8), 2399-2408 (2003).

\bibitem{Kang2024}
Kang, J., Tang, C.: Normalized solutions for the nonlinear Schr\"{o}dinger equation with potential and combined nonlinearities. Nonlinear Anal. 246, 113581 (2024).

\bibitem{Li2011}
Li, G.B., Wang, C.: The existence of a nontrivial solution to a nonlinear elliptic problem of linking type without the Ambrosetti-Rabinowitz condition. Ann. Acad. Sci. Fenn. Math. 36, 461-480 (2011).

\bibitem{Li2020}
Li, L., Pucci, P., Tang, X.: Ground state solutions for the nonlinear Schr\"{o}dinger-Bopp-Podolsky system with critical sobolev exponent. Adv. Nonlinear Stud. 20, 511-538 (2020).

\bibitem{Li2019}
Li, X., Ma, S., Zhang, G.: Existence and qualitative properties of solutions for Choquard equations with a local term. Nonlinear Anal. Real World Appl. 45, 1-25 (2019).

\bibitem{Lieb2001}
Lieb, E., Loss, M.: Analysis. American Mathematical Society, Providence (2001).

\bibitem{Lions1984-2}
Lions, P.L.: The concentration-compactness principle in the Calculus of Variation. The locally compact case, part II. Ann. Inst. H. Poincar\'{e} Anal. Non Lin\'{e}aire 1(2), 223-283 (1984).

\bibitem{LiuC2022}
Liu, C.: Existence and stability of standing waves with prescribed $L^2$-norm for a class of Schr\"{o}dinger-Bopp-Podolsky system. J. Appl. Math. Phys. 10(7), 2245-2267 (2022).

\bibitem{Liu2022}
Liu, S., Chen, H.: Existence and asymptotic behaviour of positive ground state solution for critical Schr\"{o}dinger-Bopp-Podolsky system. Electron. Res. Arch. 30, 2138-2164 (2022).

\bibitem{Mie1913}
Mie, G.: Grundlagen einer theorie der materie. Ann. Phys. 345, 1-66 (1913).

\bibitem{Moroz2013}
Moroz, V., Schaftingen, J.Van.: Groundstates of nonlinear Choquard equations: existence, qualitative properties and decay asymptotics. J. Funct. Anal. 265(2), 153-184 (2013).

\bibitem{Moroz2015}
Moroz, V., Schaftingen, J.Van.: Existence of groundstates for a class of nonlinear Choquard equations. Trans. Amer. Math. Soc. 367(9), 6557-6579 (2015).

\bibitem{Peng2024}
Peng, X.: Normalized solutions for the critical Schr\"{o}dinger-Bopp-Podolsky system. Bull. Malays. Math. Sci. Soc. 47, 1-26 (2024).

\bibitem{Podolsky1942}
Podolsky, B.: A generalized electrodynamics. I. Nonquantum. Phys. Rev. 62(2), 68-71 (1942).

\bibitem{Ramos2023}
Ramos, G., Siciliano, G.: Existence and limit behavior of least energy solutions to constrained Schr\"{o}dinger-Bopp-Podolsky systems in $\mathbb{R}^3$. Z. Angew. Math. Phys. 74(2), 56 (2023).

\bibitem{Schaftingen2008}
Schaftingen, J.Van., Willem, M.: Symmetry of solutions of semilinear elliptic problems. J. Eur. Math. Soc. 10(2), 439-456 (2008).

\bibitem{Silva2020JDE}
Silva, K.: On an abstract bifurcation result concerning homogeneous potential operators with applications to PDEs. J. Differ. Equ. 269, 7643-7675 (2020).

\bibitem{Silva2020}
Silva, K., Siciliano, G.: The fibering method approach for a non-linear Schr\"{o}dinger equation coupled with the electromagnetic field. Publ. Mat. 64, 373-390 (2020).

\bibitem{Teng2021}
Teng, K., Yan, Y.: Existence of a positive bound state solution for the nonlinear Schr\"{o}dinger-Bopp-Podolsky system. Electron. J. Qual. Theory Differ. Equ. 4, 1-19 (2021).

\bibitem{Willem1996}
Willem, M.: Minimax Theorems. Progr. Nonlinear Differential Equations Appl., vol. 24, Birkh\"{a}user Boston, Inc., Boston (1996).

\bibitem{Yang2020}
Yang, J., Chen, H., Liu, S.: The existence of nontrivial solution of a class of Schr\"{o}dinger-Bopp-Podolsky system with critical growth. Bound. Value Probl. 144, 1-16 (2020).

\bibitem{Zhu2021}
Zhu, Y., Fang, C., Chen, J.: The Schr\"{o}dinger-Bopp-Podolsky equation under the effect of nonlinearities. Bull. Malays. Math. Soc. 44, 953-980 (2021).










\end{thebibliography}
\end{document}